\documentclass[12pt]{article}

\usepackage{amsmath, amsthm, amssymb, stmaryrd, enumerate}
\usepackage[all]{xy}
\usepackage{fullpage}
\usepackage{titlesec}
\usepackage{mathrsfs}
\usepackage{graphics}

\titleformat{\section}{\normalsize\bfseries}{\thesection}{1em}{}
\titleformat{\subsection}{\normalsize\bfseries}{\thesubsection}{1em}{}

\numberwithin{equation}{subsection}

\theoremstyle{plain}

\newtheorem{proposition}[subsection]{Proposition}
\newtheorem{lemma}[subsection]{Lemma}
\newtheorem{corollary}[subsection]{Corollary}
\newtheorem{theorem}[subsection]{Theorem}

\newtheorem{proposition_sub}[subsubsection]{Proposition}
\newtheorem{lemma_sub}[subsubsection]{Lemma}

\newtheorem{theorem_sub}[subsubsection]{Theorem}

\theoremstyle{definition}

\newtheorem{definition}[subsection]{Definition}
\newtheorem{example}[subsection]{Example}
\newtheorem{remark}[subsection]{Remark}
\newtheorem{generic}[subsection]{}

\newtheorem{definition_sub}[subsubsection]{Definition}

\newtheorem{example_sub}[subsubsection]{Example}

\newcommand*{\emptybox}{\leavevmode\hbox{}}

\newcommand{\RR}{\mathbb{R}}
\newcommand{\NN}{\mathbb{N}}
\newcommand{\cInf}{C^{\infty}}
\newcommand{\cInfRn}{\cInf(\RR^n)}
\newcommand{\cInfRing}{\cInf\mbox{-ring}}
\newcommand{\Smooth}{\operatorname{\textnormal{\text{Smooth}}}}
\newcommand{\Poly}{\operatorname{\textnormal{\text{Poly}}}}
\newcommand{\Lin}{\operatorname{\textnormal{\text{Lin}}}}
\newcommand{\Set}{\operatorname{\textnormal{\text{Set}}}}
\newcommand{\Fin}{\operatorname{\textnormal{\text{Fin}}}}
\newcommand{\CinftyRing}{\ensuremath{C^\infty\kern -.5ex\operatorname{\textnormal{-\text{Ring}}}}}
\newcommand{\CMon}{\operatorname{\textnormal{CMon}}}
\newcommand{\C}{\mathcal{C}}
\newcommand{\D}{\mathcal{D}}
\newcommand{\E}{\mathcal{E}}
\newcommand{\rVec}{\operatorname{\textnormal{$\mathbb{R}$-Vec}}}
\newcommand{\SSS}{\mathsf{S}}
\newcommand{\T}{\mathsf{T}}
\newcommand{\m}{\mathsf{m}}
\newcommand{\dt}{\mathsf{d}}
\newcommand{\iT}{\mathsf{s}}
\newcommand{\algMod}{(S, \mu, \eta, \m, \mathsf{u})}
\newcommand{\Sym}{\mathsf{Sym}}
\newcommand{\di}{\partial}
\newcommand{\Lan}{\ensuremath{\operatorname{\textnormal{\textsf{Lan}}}}}
\newcommand{\f}{\textnormal{f}}
\newcommand{\cT}{\mathcal{T}}
\newcommand{\ca}[1]{\scalebox{0.85}{\raisebox{0.3ex}{$|$}} #1\scalebox{0.85}{\raisebox{0.3ex}{$|$}}}
\newcommand{\Alg}[1]{\ensuremath{#1\operatorname{\textnormal{-\text{Alg}}}}}

\begin{document}
\allowdisplaybreaks

\author{G. S. H. Cruttwell,\hspace{0.75ex} J.-S. P. Lemay,\hspace{0.75ex} R. B. B. Lucyshyn-Wright\let\thefootnote\relax
}

\title{Integral and differential structure \\ on the free $C^\infty$-ring modality}

\date{}

\maketitle

\abstract{Integral categories were recently developed as a counterpart to differential categories.  In particular, integral categories come equipped with an integration operator, known as an integral transformation, whose axioms generalize the basic integration identities from calculus such as integration by parts.  However, the literature on integral categories contains no example that captures integration of arbitrary smooth functions: the closest are examples involving integration of polynomial functions.  This paper fills in this gap by developing an example of an integral category whose integral transformation operates on smooth 1-forms.   We also provide an alternative viewpoint on the differential structure of this key example, investigate derivations and coderelictions in this context, and prove that free $C^{\infty}$-rings are Rota-Baxter algebras.}

%\abstract{This paper develops an example of an integral category whose integral transformation operates on smooth 1-forms. Further, we revisit the differential structure of this category, and we investigate derivations, coderelictions, and Rota-Baxter algebras in this context.}

\section{Introduction} \label{sec:intro}
One of the most important examples of a differential category \cite{differentialCats} captures differentiation of smooth functions by means of (co)differential structure on the free $C^\infty$-ring monad on $\RR$-vector spaces; this example was given in \cite[\S 3]{differentialCats} as an instance of a more general construction (called the $S^\infty$ construction).  It is important for at least three reasons: firstly, it is a differential category based directly on ordinary differential calculus.  Secondly, through an analogy with the role of commutative rings in algebraic geometry, $C^\infty$-rings play an important role in the semantics of \textit{synthetic differential geometry} \cite{Kock,MoeRey} and so provide a key benchmark for the generalization of aspects of commutative algebra in differential categories, including the generalizations of derivations and K\"ahler differentials in \cite{blo}.  Thirdly, the free $C^\infty$-ring monad provides a key example of a differential category that does not possess the Seely (also known as storage) isomorphisms \cite{cartDiffStorage}. Differential category structure can be simplified if one assumes the Seely isomorphisms (for more on this, see \cite{blute2018differential}); this key example shows why it is important to not assume them in general.

A recent addition to the study of categorical calculus is the story of integration and the fundamental theorems of calculus with the discovery of \emph{integral} and \emph{calculus} categories \cite{integralAndCalcCats} and differential categories with \emph{antiderivatives} \cite{integralAndCalcCats, ehrhardIntegral}. These discoveries show that both halves of calculus can be developed at this abstract categorical level. The first notion of integration in a differential category was introduced by Ehrhard in \cite{ehrhardIntegral} with the introduction of differential categories with antiderivatives, where one builds an integral structure from the differential structure. Integral categories and calculus categories were then introduced in the second author's masters thesis \cite{lemay2017integral}, under the supervision of Bauer and Cockett. Integral categories have an axiomatization of integration that is independent from differentiation, while the axioms of calculus categories describe compatibility relations between a differential structure and an integral structure via the two fundamental theorems of calculus. In particular, every differential category with antiderivatives is a calculus category. Cockett and the second author also published an extended abstract \cite{cockett_et_al:LIPIcs:2017:7668} and then a journal paper \cite{integralAndCalcCats} which provided the full story of integral categories, calculus categories, and differential categories with antiderivatives. 

However, a key example was missing in those papers: an integral category structure on the free $C^\infty$-ring monad that would be compatible with the known differential structure.  Such an example is important for the same reasons as above: it would give an integral category that resembles  ordinary calculus, and it would show that it is useful to avoid assuming the Seely isomorphisms for integral categories (noting that, as with differential categories, the assumption of the Seely isomorphisms can simplify some of the structure: for example, see \cite[Theorem 3.8]{convenientAntiDerivatives}). The journal paper on integral categories \cite{integralAndCalcCats} presented an integral category of polynomial functions, but it was not at all clear from its definition (and not known) that the formula for its deriving transformation could be generalized to yield an integral category of arbitrary smooth functions.  

Developing such an example (namely an integral category structure for the free $C^\infty$-ring monad) is the central goal of this paper.  As noted above, it is a key example whose existence demonstrates the relevance and importance of the definition of integral categories.  Moreover, in considering the integral side of this example, we have also found additional results and ideas for the differential side.  

In particular, in order to define the integral structure for this example, we have found it helpful to give an alternative presentation of its differential structure and its monad $\SSS^\infty$ on the category of vector spaces over $\RR$.  The original paper on differential categories \cite{differentialCats} did not mention the fact that $\SSS^\infty$ is the free $C^\infty$-ring monad, nor that it is a finitary monad, although it did construct this monad as an instance of a more general construction applicable for certain Lawvere theories carrying differential structure.  However, that paper \cite{differentialCats} did not define $\SSS^\infty$ by means of the usual recipe through which a finitary monad is obtained from its corresponding Lawvere theory; instead, the endofunctor $S^\infty$ was defined in \cite[\S 3]{differentialCats} by associating to each real vector space $V$ a set $S^\infty(V)$ consisting of certain mappings $h:V^* \rightarrow \RR$ on the algebraic dual $V^*$ of $V$.  

To facilitate our work with this example, we have found it helpful to give an alternative approach, via the theory of finitary monads.  Since the monad $\SSS^\infty$ is finitary, we are able to exploit standard results on locally finitely presentable categories and finitary monads to show that the differential structure carried by $\SSS^\infty$ arises by left Kan extension from structure present on the finite-dimensional real vector spaces.  Aside from shedding some new light on this important example, this approach enables us to define an integral structure on $\SSS^\infty$ through a similar method of left Kan extension, starting with integration formulae for finite-dimensional spaces.

In addition to providing a key new example of an integral category, this paper also has some further interesting aspects.  The first is in its investigation of derivations in this context. A recent paper by Blute, Lucyshyn-Wright, and O'Neill \cite{blo} defined derivations for (co)differential categories.  Here we show that derivations in this general sense, when applied to the $C^\infty$-ring example that we consider here, correspond precisely to derivations of the Fermat theory of smooth functions as defined by Dubuc and Kock \cite{dubucKock1Form}.  This provides additional evidence that the Blute/Lucyshyn-Wright/O'Neill definition is the appropriate generalization of derivations in the context of codifferential categories.  We also show that while this key example does not possess a \textit{codereliction} (see \cite{blute2018differential, differentialCats}), it does possess structure sharing many of the key features of a codereliction.  

Finally, we conclude with an interesting result on Rota-Baxter algebras.  By definition, an integral category satisfies a certain Rota-Baxter axiom.  By showing that the smooth algebra example is an integral category, we get as a corollary that free $\cInfRing$s are Rota-Baxter algebras (Proposition \ref{prop:rotaBaxterAlgebras}), a result that appears to be new.  

The paper is organized as follows.  In Section \ref{sec:background}, we review differential and integral categories, working through the definitions using the standard polynomial example.  In Section \ref{sec:fin_alg} we review and discuss some aspects of finitary monads that will be useful in relation to our central `smooth' example, including some results that are known among practitioners but whose statements we have not found to appear in the literature.  In Section \ref{sec:cinfty_modality}, we review generalities on $C^\infty$-rings, and we define the $C^\infty$-ring monad (and algebra modality) on real vector spaces.  In Section \ref{sec:diff} we define the differential structure of this example, as well as consider derivations and (co)derelictions in this context.  Finally, in Section \ref{sec:anti}, we establish the integral structure of the central example, and we conclude by proving that free $C^\infty$-rings have Rota-Baxter algebra structure.

\section{Background on differential and integral categories} \label{sec:background}

This section reviews the central structures of the paper: (co)differential categories, \mbox{(co-)integral} categories, and (co-)calculus categories \cite{differentialCats, integralAndCalcCats}.  Throughout this section, we will highlight the particular example of the category of $\mathbb{R}$-vector spaces with polynomial differentiation and integration \cite[Proposition 2.9]{differentialCats}.  While much of this material is standard, we have included it here to set a consistent notation and to clarify precisely which definitions we are using (for example, the definition of (co)differential category changed from \cite{differentialCats} to \cite{cartDiffStorage}).  

\begin{definition}
An \textbf{additive\footnote{Here we use the term \textit{additive category} to refer to any category enriched in commutative monoids, while this term is more often used for categories that are enriched in abelian groups and have finite biproducts.} symmetric monoidal category} consists of a symmetric monoidal category $(\C, \otimes, k, \sigma)$ such that $\C$ is enriched over commutative monoids and $\otimes$ preserves the commutative monoid structure in each variable separately. 
\end{definition}

\begin{example}
The category of vector spaces over $\mathbb{R}$ and $\mathbb{R}$-linear maps between them, $\rVec$, is an additive symmetric monoidal category with the structure given by the standard tensor product and the standard additive enrichment of vector spaces.
\end{example}

\begin{definition}\label{def:alg_mod}
If $(\C, \otimes, k, \sigma)$ is a symmetric monoidal category, an \textbf{algebra modality} $(\SSS,\mathsf{m},\mathsf{u})$ on $\C$ consists of:
\begin{itemize}
	\item a monad $\mathsf{S} = (S, \mu, \eta)$ on $\C$;
	\item a natural transformation $\mathsf{m}$, with components $\mathsf{m}_C: SC \otimes SC \to SC$ $\:(C \in \C)$;
	\item a natural transformation $\mathsf{u}$, with components $\mathsf{u}_C: k \to SC$ $\:(C \in \C)$;
\end{itemize}
such that
\begin{itemize}
	\item for each object $C$ of $\C$, $(SC,\mathsf{m}_C, \mathsf{u}_C)$ is a commutative monoid (in the symmetric monoidal category $\C$);
	\item each component of $\mu$ is a monoid morphism (with respect to the obvious monoid structures).  
\end{itemize}
Such an algebra modality $(\SSS,\mathsf{m},\mathsf{u})$ will also be denoted by $\algMod$ or by $\SSS$.
\end{definition}

\begin{example}\label{exa:sym_on_rvec} $\rVec$ has an algebra modality $\Sym$, which sends a vector space $V$ to the symmetric algebra on $V$ (over $\RR$),
	\[ \Sym(V) := \bigoplus_{n=0}^{\infty} \Sym^{n}(V) \]
where $\Sym^{0}(V) := \mathbb{R}, \Sym^{1}(V) := V$, and for $n \geq 2$, $\Sym^{n}(V)$ is the quotient of the tensor product of $V$ with itself $n$ times by the equations
	\[ v_1 \otimes \ldots \otimes v_i \otimes \ldots v_n = v_{\sigma(1)} \otimes \ldots \otimes v_{\sigma(i)} \otimes \ldots \otimes v_{\sigma(n)} \]
associated to permutations $\sigma$ of $\{1, 2, \ldots n\}$.  It is a standard result that $\Sym(V)$ can also be identified with a polynomial ring: if $X = \{x_i \mid i \in I\}$ is a basis for $V$, then
	\[ \Sym(V) \cong \mathbb{R}[X]. \]
Also, $\Sym(V)$ is the free commutative $\RR$-algebra on the $\RR$-vector space $V$.
\end{example}

\begin{definition}\label{defn:codiff}
If $(\C, \otimes, k, \sigma)$ is an additive symmetric monoidal category with an algebra modality $\algMod$, then a \textbf{deriving transformation} on $\C$ is a natural transformation $\dt$, with components
    \[ \dt_C: SC \to SC \otimes C \;\;\;\;\;\;(C \in \C)\]
such that\footnote{Note that here, and throughout, we denote diagrammatic (left-to-right) composition by juxtaposition, whereas we denote right-to-left, non-diagrammatic composition by $\circ$, and functions $f$ are applied on the left, parenthesized as in $f(x)$; however, we write composition of functors in the right-to-left, non-diagrammatic order, and functors $F$ are applied on the left, as in $FX$.  We suppress the use of the monoidal category associator and unitor isomorphisms, and we omit subscripts and whiskering on the right.} 
\begin{enumerate}[{\bf [d.1]}]
    \item \textbf{Derivative of a constant}: $\mathsf{u}\dt = 0$;
    \item \textbf{Leibniz/product rule}: $\m \dt = [(1 \otimes \dt)(\m \otimes 1)] + [(\dt \otimes 1)(1 \otimes \sigma)(\m \otimes 1)]$;
    \item \textbf{Derivative of a linear function}: $\eta \dt = \mathsf{u} \otimes 1$;
    \item \textbf{Chain rule}: $\mu \dt = \dt (\mu \otimes \dt)(\m \otimes 1)$;
    \item \textbf{Interchange}\footnote{This rule was not in the original paper \cite{differentialCats}, but was later formally introduced in \cite{cartDiffStorage}, and is used in \cite{integralAndCalcCats}.  It represents the independence of order of partial differentiation.}: $\dt(\dt \otimes 1) = \dt (\dt \otimes 1)(1 \otimes \sigma)$.
\end{enumerate}
Such a $\C$ equipped with a deriving transformation $\dt$ is called a \textbf{codifferential category}.  
\end{definition}

\begin{example}\label{ex:rVecDiff}
$\rVec$ is a codifferential category with respect to the deriving transformation $\dt_V: \Sym(V) \to \Sym(V) \otimes V$ defined on pure tensors by
    \[ \dt_V(a_1 \otimes \ldots \otimes a_n) := \sum_{i=1}^n (a_1 \otimes \ldots \otimes a_{i-1} \otimes a_{i+1} \otimes \ldots \otimes a_n) \otimes a_i \]
where $a_1,...,a_n \in V$.  If $V$ has a basis $X$, then with respect to the isomorphism $\Sym(V) \cong \mathbb{R}[X]$, $\dt_V: \mathbb{R}[X] \to \mathbb{R}[X] \otimes V$ is given by taking a sum involving the partial derivatives:
    \[ \dt_V(x_1^{n_1} \ldots x_k^{n_k}) = \sum_{i=1}^k n_i \cdot x_1^{n_1} \ldots x_i^{n_i-1} \ldots x_k^{n_k} \otimes x_i. \]
\end{example}

We now turn to the integral side of this theory, as described in \cite{integralAndCalcCats}.  

\begin{definition}\label{defn:coint}
If $(\C, \otimes, k, \sigma)$ is an additive symmetric monoidal category with an algebra modality $\algMod$, then an \textbf{integral transformation} on $\C$ is a natural transformation $\iT$, with components
    \[ \iT_C: SC \otimes C \to SC \;\;\;\;\;\;(C \in \C)\]
such that
\begin{enumerate}[{\bf [s.1]}]
\item \textbf{Integral of a constant}: $(\mathsf{u} \otimes 1)\iT=\eta$;
\item \textbf{Rota-Baxter rule}: $(\iT \otimes \iT)\m  = [(\iT \otimes 1 \otimes 1)(\m \otimes 1)\iT]+[(1 \otimes 1 \otimes \iT)(1 \otimes \sigma)(\m \otimes 1)\mathsf{s}$];
\item \textbf{Interchange}: $(\iT \otimes 1)\iT =(1 \otimes \sigma)(\iT \otimes 1)\iT$.
\end{enumerate}
Such a $\C$ equipped with an integral transformation $\iT$ is called a \textbf{co-integral category}. 
\end{definition}

We refer the reader to \cite{integralAndCalcCats} for more intuition on these axioms, particularly if the reader is interested in the Rota-Baxter rule: roughly, it is ``integration by parts'', but expressed solely in terms of integrals.  We discuss the Rota-Baxter rule more in Section \ref{sec:anti} when discussing Rota-Baxter algebras.  

\begin{example}\label{ex:rVecInt}
$\rVec$ is a co-integral category, with integral transformation $\iT_V: \Sym(V) \otimes V \to \Sym(V)$ defined on generators by
    \[ \iT_V((v_1 \otimes \ldots \otimes v_n) \otimes w) := \frac{1}{n+1} v_1 \otimes \ldots \otimes v_n \otimes w\;\;\;\;\;\;(v_1,...,v_n,w \in V). \]
If $V$ has a basis $X$, then with respect to the isomorphism $\Sym(V) \cong \mathbb{R}[X]$, $\iT_V: \mathbb{R}[X] \otimes V \to \mathbb{R}[X]$ is given by
    \[ \iT_V((x_1^{n_1} \ldots x_k^{n_k}) \otimes x_i) = \frac{1}{1 + \sum_{i=1}^k n_k} x_1^{n_1} \ldots x_i^{n_i + 1} \ldots x^{n_k}. \]
\end{example}
Note that the form the integral transformation takes in this example is perhaps slightly unexpected: the denominator sums \emph{all} of the exponents in the monomial, not just the exponent on the indeterminate with respect to which integration occurs.  As noted in \cite{integralAndCalcCats}, ``at first glance this may seem bizarre ... however, [simply taking $n_i + 1$] fails the Rota-Baxter rule for any vector space of dimension greater than one''.  We shall see in this paper, however, a more abstract reason why this is the right integral transformation on polynomials: it can be recovered from the integral transformation for smooth functions, by restricting to polynomials (see Remark \ref{rmk:integralOfPolys}).  Thinking about it another way, it is not at all clear how to extend the above formula to an arbitrary smooth function; one of the accomplishments of the present work is in finding this extension.  

We now consider categories with differential and integral structure that are compatible (in the sense of the fundamental theorems of calculus).  

\begin{definition}\label{cocaldef} A \textbf{co-calculus category} \cite{integralAndCalcCats} is a codifferential category and a co-integral category on the same algebra modality such that the deriving transformation $\dt$ and the integral transformation $\iT$ satisfy the following: 
\begin{enumerate}[{\bf [c.1]}]
\item The \textbf{Second Fundamental Theorem of Calculus}:  $\dt\iT + S(0)=1$;
\item The \textbf{Poincar\' e condition}: if $f: B \to S(C) \otimes C$ is such that
\[f(\dt \otimes 1)(1 \otimes \sigma)=f(\dt \otimes 1)\]
then $f$ \textbf{satisfies the First Fundamental Theorem}; that is, $f \iT \dt=f$.  
\end{enumerate}
\end{definition}

\begin{remark} The axioms of a calculus category were first described by Ehrhard in \cite{ehrhardIntegral} as consequences of his notion of a differential category with antiderivatives. 
\end{remark}

\begin{example}
$\rVec$, with the `polynomial' codifferential and co-integral structure carried by the symmetric algebra monad $\Sym$ (\ref{ex:rVecDiff}, \ref{ex:rVecInt}), is a co-calculus category.  
\end{example}

In fact, $\rVec$ is even stronger: it is a (co)differential category with antiderivatives.  Before defining this notion, we first need to recall certain natural transformations associated with algebra modalities and deriving transformations.  

\begin{definition}\label{def:coderiving} The \textbf{coderiving transformation} \cite{integralAndCalcCats} for an algebra modality $\algMod$ is the natural transformation $\dt^\circ_A: SA \otimes A \to SA$ defined as follows: 
  \[\dt^\circ :=  (1 \otimes \eta)\m\]
 \end{definition}
 
As discussed in \cite{integralAndCalcCats}, while the coderiving transformation is of the same type as an integral transformation, in most cases it is \textbf{NOT} an integral transformation. However, it is used in the construction of the integral transformation for a differential category with antiderivatives.  
 
 \begin{definition}\label{def:K} For a codifferential category with algebra modality $(\mathsf{S}, \mu, \eta, \mathsf{m}, \mathsf{u})$ and deriving transformation $\mathsf{d}$, define the following natural transformations  \cite{integralAndCalcCats}, all of type $S \Rightarrow S$:
\begin{enumerate}[{\em (i)}]
\item $\mathsf{L} :=  \dt\dt^{\circ}$
\item $\mathsf{K}:= \mathsf{L} + S(0)$
\item $\mathsf{J} := \mathsf{L} + 1$. 
\end{enumerate} 
A codifferential category is said to have \textbf{antiderivatives} if $\mathsf{K}$ is a natural isomorphism.
\end{definition}

In \cite{ehrhardIntegral} Ehrhard uses a slightly different definition of having antiderivatives, instead of asking that $\mathsf{J}$ be invertible. However, as shown in \cite[Proposition 6.1]{integralAndCalcCats}, the invertibility of $\mathsf{K}$ implies that of $\mathsf{J}$.  Moreover, if $K$ or $J$ is invertible, then one can construct a co-integral category with integral transformation constructed using either  $\mathsf{K}^{-1}$ or $\mathsf{J}^{-1}$, and the two constructions give the same result when both are invertible. The reason to use $\mathsf{K}$ over $\mathsf{J}$ is that $\mathsf{K}$ being invertible immediately implies one has a co-calculus category. On the other hand, while $\mathsf{J}$ being invertible gives a co-integral category, one needs an added condition (known as the Taylor Property \cite[Definition 5.3]{integralAndCalcCats}) to also obtain a co-calculus category. 

\begin{theorem}\label{antithm} \textnormal{\cite{integralAndCalcCats}} A codifferential category with antiderivatives is a co-calculus category whose integral transformation is defined by $\iT := \dt^\circ \mathsf{K}^{-1}= (\mathsf{J}^{-1} \otimes 1)\dt^\circ$.
\end{theorem} 

\begin{example}
With the structure of polynomial differentiation given above, $\rVec$ is a codifferential category with antiderivatives, and its integral transformation is of the form given in the theorem above \cite{integralAndCalcCats}. Indeed, in this case one finds that $\mathsf{K}_V$ is the identity on scalars and scalar multiplies a pure tensor $v_1 \otimes \ldots \otimes v_n$ by $n$:
    \[ \mathsf{K}_V(v_1 \otimes \ldots \otimes v_n) = n \cdot (v_1 \otimes \ldots \otimes v_n) \]
while $\mathsf{J}_V$ is also the identity on scalars but instead scalar multiplies $v_1 \otimes \ldots \otimes v_n$ by $n+1$:
\[ \mathsf{J}_V(v_1 \otimes \ldots \otimes v_n) = (n+1) \cdot (v_1 \otimes \ldots \otimes v_n). \]    
$\mathsf{K}$ is clearly invertible, and therefore so is $\mathsf{J}$, and one can calculate that the resulting integral transformation $\iT := \dt^\circ \mathsf{K}^{-1} = (\mathsf{J}^{-1} \otimes 1)\dt^\circ$ is the one given above in Example \ref{ex:rVecInt}.    
\end{example}

Many more examples of (co)differential and (co-)integral categories can be found in \cite[\S 7]{integralAndCalcCats}.  Our main focus in this paper is the differential and integral structure of arbitrary smooth functions.

\section{Some fundamentals of finitary algebra}\label{sec:fin_alg}

In Section \ref{sec:cinfty_modality}, we shall give a construction of a particular algebra modality $\SSS^\infty$ on the category $\rVec$ of real vector spaces, such that the category of $\SSS^\infty$-algebras is the category of $C^\infty$-rings.  The monad $\SSS^\infty$ is finitary, and so in the present section we shall first review and discuss some basics on finitary monads and Lawvere theories, which will provide the basis of our approach to defining $\SSS^\infty$ and equipping it with further structure.  While much of this material is standard, we also discuss certain results that are known among practitioners but whose statements we have not found to appear in the literature, such as Propositions \ref{thm:charn_fmonadic} and \ref{thm:cancellation_lemma}. 

\subsection{Finitary monads on locally finitely presentable categories}\label{sec:fin_lfp}

Let us recall that an object $C$ of a locally small category $\C$ is \textbf{finitely presentable} if the functor $\C(C,-):\C \rightarrow \Set$ preserves filtered colimits.  Here, following \cite{Kelly}, we use the term \textit{filtered colimit} to mean the colimit of a functor whose domain is not only filtered but also small\footnote{Again following \cite{Kelly}, we call a category \textit{small} if it has but a (small) set of isomorphism classes.}.  We denote by $\C_\f$ the full subcategory of $\C$ consisting of the finitely presentable objects.  Recall that $\C$ is \textbf{locally finitely presentable} (\textbf{l.f.p.}) iff $\C$ is cocomplete and the full subcategory $\C_\f$ is small and dense (in $\C$) \cite[Corollary 7.3]{Kelly}.

\begin{example_sub}\label{exa:rvec_lfp} $\rVec$ is l.f.p., and a vector space is finitely presentable if and only if it is finite-dimensional. Therefore $\rVec_\f$ is equivalent to the category $\Lin_\RR$ whose objects are the cartesian spaces $\RR^n$ and whose morphisms are arbitrary $\RR$-linear maps between these spaces.
\end{example_sub}

A functor between l.f.p.\ categories is \textbf{finitary} if it preserves filtered colimits.  Letting $\C$ be an l.f.p.\ category, a \textbf{finitary monad} on $\C$ is a monad on $\C$ whose underlying endofunctor is finitary.  By \cite[Proposition 7.6]{Kelly}, we have the following well-known result, which will be of central importance to us:

\begin{proposition_sub}\label{prop:equiv_fin}
Let $\C$ and $\D$ be l.f.p.\ categories, and let $\iota:\C_\f \hookrightarrow \C$ denote the inclusion.  Then there is an equivalence of categories
\begin{equation}\label{eq:equiv_fin}
\xymatrix {
[\C_\f,\D] \ar@/_0.5pc/[rr]_{\Lan_\iota}^{\sim} & & \Fin(\C,\D) \ar@/_0.5pc/[ll]_{\iota^*}
}
\end{equation}
between the category $[\C_\f,\D]$ of functors from $\C_\f$ to $\D$ and the category $\Fin(\C,\D)$ of finitary functors from $\C$ to $\D$.  The functor $\iota^*$ is given by restriction along $\iota$, and its pseudo-inverse $\Lan_\iota$ is given by left Kan-extension along $\iota$.  Furthermore, a functor $F:\C \rightarrow \D$ is finitary if and only if it is a left Kan extension along $\iota$, if and only if it is a left Kan extension of $F\iota$ along $\iota$.
\end{proposition_sub}

In this paper, we shall be concerned with the case of Proposition \ref{prop:equiv_fin} where $\D = \C$ for an l.f.p.\ category $\C$, in which case we have an equivalence $[\C_\f,\C] \simeq \Fin(\C,\C)$.  As described in \cite[\S 4]{KellyPower}, the category $[\C_\f,\C]$ carries a monoidal product for which the equivalence
$$[\C_\f,\C] \simeq \Fin(\C,\C)$$
is monoidal, so that finitary monads on $\C$ may be described equivalently as monoids in $[\C_\f,\C]$.

\begin{example_sub}\label{exa:rvec2}
Recalling that $\rVec$ is l.f.p. and $\rVec_\f \simeq \Lin_\RR$ (Example \ref{exa:rvec_lfp}), we have an equivalence
\begin{equation}\label{eq:equiv_rvec_fin}
\xymatrix {
[\Lin_\RR,\rVec] \ar@/_0.5pc/[rr]_{\Lan_\iota}^{\sim} & & \Fin(\rVec,\rVec) \ar@/_0.5pc/[ll]_{\iota^*}
}
\end{equation}
given by restriction and left Kan extension along the inclusion $\iota:\Lin_\RR \hookrightarrow \rVec$.  In \S \ref{sec:cinfty_modality}, we will define a finitary monad on $\rVec$ whose corresponding functor $\Lin_\RR \rightarrow \rVec$ sends $\RR^n$ to the space $C^\infty(\RR^n)$ of smooth, real-valued functions on $\RR^n$.
\end{example_sub}

\begin{proposition_sub}\label{thm:tens_pr_fin}
Let $F,G:\C \rightarrow \D$ be finitary functors between l.f.p.\ categories $\C$ and $\D$, and suppose that $\D$ is equipped with a functor $\otimes:\D \times \D \rightarrow \D$ that preserves filtered colimits in each variable separately.  Then the point-wise tensor product $F \otimes G = F(-)\otimes G(-):\C \rightarrow \D$ is finitary.
\end{proposition_sub}
\begin{proof}
$F \otimes G$ is the composite $\C \xrightarrow{\langle F, G\rangle} \D \times \D \xrightarrow{\otimes} \D$, and since $F$ and $G$ are finitary and colimits in $\D \times D$ are point-wise, it follows that $\langle F, G \rangle$ preserves filtered colimits.  Hence it suffices to show that $\otimes$ preserves filtered colimits.  Every filtered colimit in $\D \times \D$ is of the form $\varinjlim \langle D,E\rangle = (\varinjlim D,\varinjlim E)$ for functors $D,E:\mathcal{J} \rightarrow \D$ on a small, filtered category $\mathcal{J}$, and our assumption on $\otimes$ entails that
$$(\varinjlim D) \otimes (\varinjlim E) \cong \varinjlim_{J \in \mathcal{J}}\varinjlim_{K \in \mathcal{J}} DJ \otimes EK \cong \varinjlim_{J \in \mathcal{J}} DJ \otimes EJ\;,$$
since the diagonal functor $\Delta:\mathcal{J} \rightarrow \mathcal{J} \times \mathcal{J}$ is final as $\mathcal{J}$ is filtered \cite[2.19]{AdamekRosickyVitale}.
\end{proof}

Given categories $\C$ and $\D$ and a functor $G:\D \rightarrow \C$, we shall say that $G$ is \textbf{strictly monadic} if $G$ has a left adjoint such that the comparison functor $\D \rightarrow \C^\T$ is an isomorphism, where $\C^\T$ denotes the category of algebras of the induced monad $\T$ on $\C$.  Supposing that $\C$ is l.f.p., let us say that $G$ is \textbf{strictly finitary monadic} if $G$ is strictly monadic and the induced monad on $\C$ is finitary.  In the latter case, since $\C^\T$ is necessarily l.f.p \cite[Ch. 3]{AdamekRosicky}, it then follows that $\D$ is l.f.p.\ also.

We shall require the following characterizations of categories of algebras of finitary monads on a given l.f.p.\ category.  Given a functor $G:\D \rightarrow \C$, we shall say that a parallel pair of morphisms $f,g$ in $\D$ is a \textbf{$G$-absolute pair} if the pair $Gf,Gg$ has an absolute coequalizer in $\C$.

\begin{proposition_sub}\label{thm:charn_fmonadic}
Let $\C$ and $\D$ be l.f.p.\ categories, and let $G:\D \rightarrow \C$ be a functor.  Then the following are equivalent:
\begin{enumerate}
  \item $G$ is strictly finitary monadic;
  \item $G$ creates small limits, filtered colimits, and coequalizers of $G$-absolute pairs;
  \item $G$ preserves small limits and filtered colimits, and $G$ creates coequalizers of $G$-absolute pairs.
\end{enumerate}
\end{proposition_sub}
\begin{proof}
Suppose (1).  Then $G$ creates limits \cite[Proposition 4.3.1]{Borceux:Handbook2}, and since the induced monad $T$ preserves filtered colimits it follows that $G$ creates filtered colimits \cite[Proposition 4.3.2]{Borceux:Handbook2}.  Hence (2) holds, by Beck's Monadicity Theorem \cite[III.7, Thm. 1]{MacLane}.

Since $\C$ is not only cocomplete but also complete \cite[1.28]{AdamekRosicky}, the creation of small limits and filtered colimits by $G$ entails their preservation, so (2) implies (3).

Lastly suppose (3).  Then we deduce by \cite[1.66]{AdamekRosicky} that $G$ has a left adjoint, and we deduce by Beck's Monadicity Theorem \cite[III.7, Thm. 1]{MacLane} that  $G$ is strictly monadic.  But since $G$ preserves filtered colimits and its left adjoint $F$ preserves arbitrary colimits, it follows that the induced monad $T = GF$ preserves filtered colimits, so (1) holds.
\end{proof}

\begin{proposition_sub}\label{thm:cancellation_lemma}
Let $\C,\D,\E$ be l.f.p.\ categories, and suppose that we are given a commutative diagram of functors
$$
\xymatrix{
\E \ar[rr]^{U} \ar[dr]_H &      & \D \ar[dl]^{G}\\
                         & \C &
}
$$
in which $H$ and $G$ are strictly finitary monadic.  Then $U$ is strictly finitary monadic.
\end{proposition_sub}
\begin{proof}
Let us say that a functor $F$ \textit{preserves} (resp. \textit{creates}) if $F$ preserves (resp. creates) small limits, filtered colimits, and coequalizers of $F$-absolute pairs.  By Proposition \ref{thm:charn_fmonadic}, both $G$ and $H = GU$ preserve and create, so it follows by a straightforward argument that $U$ creates.  The result now follows from \ref{thm:charn_fmonadic}.
\end{proof}

\subsection{Some basics on Lawvere theories}\label{sec:LTh}

By definition, a \textbf{Lawvere theory} \cite{Law:PhD} is a small category $\cT$ with a denumerable set of distinct objects $T^0$, $T^1$, $T^2$, $\ldots$ in which each object $T^n$ $(n \in \NN)$ is equipped with a family of morphisms $(\pi_i:T^n \rightarrow T)_{i = 1}^n$ that present $T^n$ as an $n$-th power of the object $T = T^1$.  We can and will assume that the given morphism $\pi_1:T^1 \rightarrow T$ is the identity morphism.

\begin{example_sub}\label{exa:poly}
There is a Lawvere theory $\Poly_\RR$ whose objects are the cartesian spaces $\RR^n$ $(n \in \NN)$ and whose morphisms $p:\RR^n \rightarrow \RR^m$ are \textit{algebraic mappings}, i.e. maps $p = (p_1,\ldots,p_m)$ whose coordinate functions $p_j:\RR^n \rightarrow \RR$ $(j = 1, \ldots, m)$ are polynomial functions; equivalently, we may describe the morphisms of $\Poly_\RR$ as $m$-tuples of formal polynomials in $n$ variables. 
\end{example_sub}

\begin{example_sub}\label{exa:lin}
There is a Lawvere theory $\Lin_\RR$ whose objects are the same as those of $\Poly_\RR$ (\ref{exa:poly}), but whose morphisms $M:\RR^n \rightarrow \RR^m$ are \textit{$\RR$-linear maps}, which we shall identify with their corresponding $m \times n$ matrices.
\end{example_sub}

Given a Lawvere theory $\cT$, a \textbf{$\cT$-algebra} is a functor $A:\cT \rightarrow \Set$ that preserves finite powers (or, equivalently, preserves finite products).  Every $\cT$-algebra $A$ has an underlying set $\ca{A} = A(T)$, and for each $n$ the set $A(T^n)$ is an $n$-th power of the set $\ca{A}$.  Writing $\ca{A}^n$ to denote the usual choice of $n$-th power of $\ca{A}$, i.e. the set of $n$-tuples of elements of $\ca{A}$, we say that a $\cT$-algebra $A$ is \textit{normal} if $A$ sends each of the given power cones $(\pi_i:T^n \rightarrow T)_{i = 1}^n$ to the usual $n$-th power cone $(\pi_i:\ca{A}^n \rightarrow \ca{A})_{i = 1}^n$ (\cite[Definition 5.10]{Lu:EnrAlgTh}, \cite[2.4]{Lu:CvxAffCmt}).

$\cT$-algebras are the objects of a category in which the morphisms are natural transformations, and this category has an equivalent full subcategory consisting of the normal $\cT$-algebras (\cite[Theorem 5.14]{Lu:EnrAlgTh}, \cite[2.5]{Lu:CvxAffCmt}).

The category of normal $\cT$-algebras is equipped with a `forgetful' functor to $\Set$, given by evaluating at $T$, and this functor is strictly finitary monadic, so the category of normal $\cT$-algebras is \textit{isomorphic}\footnote{The category of \textit{all} $\cT$-algebras is merely \textit{equivalent} to the category of $\T$-algebras.} to the category of $\T$-algebras for an associated finitary monad $\T$ on $\Set$; e.g. see \cite[2.6]{Lu:CvxAffCmt}.  

\begin{example_sub}\label{exa:ralg}
The category of normal $\Poly_\RR$-algebras for the Lawvere theory $\Poly_\RR$ in Example \ref{exa:poly} is isomorphic to the category $\Alg{\RR}$ of commutative $\RR$-algebras (e.g. by\footnote{It is well known that the category of \textit{all} $\Poly_\RR$-algebras is (merely) \textit{equivalent} to the category of commutative $\RR$-algebras.} \cite[2.9]{Lu:CvxAffCmt}).
\end{example_sub}

\begin{example_sub}\label{exa:rvec}
The category of normal $\Lin_\RR$-algebras for the Lawvere theory $\Lin_\RR$ in \ref{exa:lin} is isomorphic to the category $\rVec$ of $\RR$-vector spaces (e.g. by\footnote{It is well known that the category of \textit{all} $\Lin_\RR$-algebras is (merely) \textit{equivalent} to $\rVec$.} \cite[2.8]{Lu:CvxAffCmt}).
\end{example_sub}

\section{The free $C^\infty$-ring modality on vector spaces}\label{sec:cinfty_modality}

There is a Lawvere theory $\Smooth$ whose objects are the cartesian spaces $\RR^n$ $(n \in \NN)$ and whose morphisms are arbitrary smooth maps between them.  By a \textbf{$C^\infty$-ring} we shall mean a normal $\Smooth$-algebra\footnote{More often, a $C^\infty$-ring is defined as an arbitrary $\Smooth$-algebra, but with the above definition we obtain an equivalent category, and one that is \textit{strictly finitary monadic} over $\Set$ (\S \ref{sec:fin_lfp}, \ref{sec:LTh}) and so \textit{isomorphic} (rather than just equivalent) to a variety of algebras in Birkhoff's sense \cite[III.8]{MacLane}.}.  Hence $C^\infty$-rings are the objects of a category $\CinftyRing$, the category of normal $\Smooth$-algebras (\S \ref{sec:LTh}).

With this definition, a $C^\infty$-ring $A:\Smooth \rightarrow \Set$ is uniquely determined by its \textit{underlying set} $X = A(\RR)$ and the mappings $\Phi_f = A(f):X^m = A(\RR^m) \rightarrow A(\RR) = X$ associated to smooth, real-valued functions $f \in C^\infty(\RR^m)$ $(m \in \NN)$.  Hence $A$ may be described equivalently as a pair $(X,\Phi)$ consisting of a set $X$ and a suitable family of mappings $\Phi_f$ of the above form, called \textit{operations}, satisfying certain conditions; this notation is as in \cite{joyce2011introduction}, where the resulting conditions on $\Phi$ are also stated explicitly.  A morphism of $C^\infty$-rings $\phi:(X,\Phi) \rightarrow (Y,\Psi)$ is given by a mapping $\phi:X \rightarrow Y$ that preserves all of the operations $\Phi_f$, $\Psi_f$, in the evident sense.

Note that there is a faithful inclusion
$$\Poly_\RR \hookrightarrow \Smooth,$$
where $\Poly_\RR$ is the Lawvere theory considered in Example \ref{exa:poly}. This inclusion functor induces a functor from the category of normal $\Smooth$-algebras to the category of normal $\Poly_\RR$-algebras, given by pre-composition.  In other words, we obtain a functor $\CinftyRing \rightarrow \Alg{\RR}$, so that every $C^\infty$-ring carries the structure of a commutative $\RR$-algebra. Moreover, we have a commutative diagram of faithful inclusions
$$
\xymatrix{
          & \Lin_\RR \ar@{^(->}[dl] \ar@{^(->}[dr] & \\
\Poly_\RR \ar@{^(->}[rr] &                         & \Smooth
}
$$
where $\Lin_\RR$ is the Lawvere theory considered in Example \ref{exa:lin}. These inclusions induce a commutative diagram of functors
\begin{equation}\label{eq:alg_functors}
\xymatrix{
\CinftyRing \ar[rr] \ar[dr]_U &      & \Alg{\RR} \ar[dl]^V\\
                              & \rVec &
}
\end{equation}
between the categories of normal algebras of these Lawvere theories, where we identify $\rVec$ and $\Alg{\RR}$ with the categories of normal $\Lin_\RR$-algebras and $\Poly_\RR$-algebras, respectively (Example \ref{exa:rvec}, Example \ref{exa:ralg}).  

The functor $U$ in \eqref{eq:alg_functors} participates in a commutative diagram
$$
\xymatrix{
\CinftyRing \ar[rr]^{U} \ar[dr]_H &      & \rVec \ar[dl]^{G}\\
                         & \Set &
}
$$
in which the forgetful functors $H$ and $G$ are strictly finitary monadic (by \S \ref{sec:LTh}).  Hence by Theorem \ref{thm:cancellation_lemma} we deduce the following result:

\begin{proposition}
There is a strictly finitary monadic functor $U:\CinftyRing \rightarrow \rVec$ that sends each $C^\infty$-ring $A$ to its underlying $\RR$-vector space (which we denote also by $A$).
\end{proposition}

\begin{definition}
We denote by $\SSS^\infty = (S^\infty,\mu,\eta)$ the finitary monad on $\rVec$ induced by the strictly finitary monadic functor $U:\CinftyRing \rightarrow \rVec$.  We call $\SSS^\infty$ the \textbf{free $C^\infty$-ring monad on the category of real vector spaces}.
\end{definition}
\begin{corollary}
The category $\CinftyRing$ of $C^\infty$-rings is isomorphic to the category $\rVec^{\SSS^\infty}$ of $\SSS^\infty$-algebras for the finitary monad $\SSS^\infty$ on $\rVec$.
\end{corollary}

We may of course apply similar reasoning to the functor $V:\Alg{\RR} \rightarrow \rVec$ in \eqref{eq:alg_functors}, thus deducing also that $V$ is strictly finitary monadic.  The induced monad $\Sym$ on $\rVec$ is described in Example \ref{exa:sym_on_rvec}.  Hence we may make the following identifications:
\begin{equation}\label{eqn:ralg_and_cinfty_ring}\Alg{\RR} = \rVec^{\Sym},\;\;\;\;\;\;\;\;\CinftyRing = \rVec^{\SSS^\infty}\;.\end{equation}

\begin{generic}\label{par:free_cinfty_on_fin_set}
Letting $n \in \NN$, it is well known that the set $C^\infty(\RR^n)$ of all smooth, real-valued functions on $\RR^n$ underlies the free $C^\infty$-ring on $n$ generators, i.e., the free $C^\infty$-ring on the set $\{1,2,...,n\}$ \cite{MoeRey}.  The operations
$$\Phi_g\;:\;(C^\infty(\RR^n))^m \rightarrow C^\infty(\RR^n)\;\;\;\;\;\;\;\;(g \in \C^\infty(\RR^m))$$
carried by this $C^\infty$-ring are given by 
\[\Phi_g(f_1,...,f_m) = g \circ \langle f_1,...,f_m \rangle\]
 where $\circ$ denotes right-to-left, non-diagrammatic composition.  The projections $\pi_i \in C^\infty(\RR^n)$ $(i=1,...,n)$ serve as generators, in the sense that the mapping $\pi_{(-)}:\{1,2,...,n\} \rightarrow C^\infty(\RR^n)$ given by $i \mapsto \pi_i$ presents this $C^\infty$-ring as free on the set $\{1,2,...,n\}$.  Given a mapping $a:\{1,2,...,n\} \rightarrow A$ valued in a $C^\infty$-ring $(A,\Psi)$, the unique morphism of $C^\infty$-rings $a':C^\infty(\RR^n) \rightarrow A$ such that $\pi_{(-)}a' = a$ is given by $a'(g) = \Psi_g(a(1),...,a(n))$.  From this we obtain the following:
\end{generic}

\begin{proposition}\label{thm:free_cinfty}
The free $C^\infty$-ring on the vector space $\RR^n$ $(n \in \NN)$ is $C^\infty(\RR^n)$, with operations as described above.  The unit morphism $\eta_{\RR^n}:\RR^n \rightarrow C^\infty(\RR^n)$ sends the standard basis vectors $e_1,...,e_n \in \RR^n$ to the projection functions $\pi_1,...,\pi_n$.  Given any linear map $\phi:\RR^m \rightarrow A$ valued in a $C^\infty$-ring $(A,\Psi)$, there is a unique morphism of $C^\infty$-rings $\phi^\sharp:C^\infty(\RR^n) \rightarrow A$ such that $\eta_{\RR^n}\phi^\sharp = \phi$, given by
$$\phi^\#(g) = \Psi_g(\phi(e_1),...,\phi(e_n))\;\;\;\;\;\;(g \in C^\infty(\RR^n))\;.$$
\end{proposition}
\begin{proof}
The vector space $\RR^n$ is free on the set $\{1,2,...,n\}$, so this follows from \ref{par:free_cinfty_on_fin_set}.
\end{proof}

\begin{remark}\label{rem:cinf_sinf}
By applying Proposition \ref{thm:free_cinfty} and choosing the left adjoint to $U$ suitably, we can and will assume that
$$S^\infty(\RR^n) = C^\infty(\RR^n)\;.$$
Accordingly, we will denote the restriction of $S^\infty$ along the inclusion $\iota:\Lin_\RR \hookrightarrow \rVec$ by
$$C^\infty = S^\infty \iota\;\;:\;\;\Lin_\RR \longrightarrow \rVec\;.$$
Hence, since $S^\infty$ is finitary, we deduce by Proposition \ref{prop:equiv_fin} and Example \ref{exa:rvec2} that $S^\infty$ is a left Kan extension of $C^\infty:\Lin_\RR \rightarrow \rVec$ along $\iota$.  Symbolically,
$$S^\infty = \Lan_\iota C^\infty\;.$$
Hence
$$S^\infty(V) \;\;\;\;\cong \varinjlim_{(\RR^n,\phi) \:\in\: \Lin_\RR \slash V}C^\infty(\RR^n)$$
naturally in $V \in \rVec$, where $\Lin_\RR\slash V$ denotes the comma category whose objects are pairs $(\RR^n,\phi)$ consisting of an object $\RR^n$ of $\Lin_\RR$ and a morphism $\phi:\RR^n \rightarrow V$ in $\rVec$.  Equivalently, the maps
$$S^\infty(\phi)\;:\;S^\infty(\RR^n) = C^\infty(\RR^n) \longrightarrow S^\infty(V)\;,\;\;\;\;\;\;(\RR^n,\phi) \in \Lin_\RR\slash V\;,$$
present $S^\infty(V)$ as a colimit of the composite functor $\Lin_\RR\slash V \xrightarrow{\pi} \Lin_\RR \xrightarrow{C^\infty} \rVec$ (where $\pi$ is the forgetful functor).
\end{remark}

\begin{proposition}\label{thm:charn_cinfty}
The functor $C^\infty$ sends each $\RR$-linear map $h:\RR^n \rightarrow \RR^m$ to the map $C^\infty(h):\C^\infty(\RR^n) \rightarrow C^\infty(\RR^m)$ that sends each $g \in C^\infty(\RR^n)$ to the composite
$$\RR^m \xrightarrow{h^*} \RR^n \xrightarrow{g} \RR$$
where $h^*$ denotes the transpose (or adjoint) of $h$.
\end{proposition}
\begin{proof}
By definition $C^\infty$ sends $h$ to the unique $C^\infty$-ring morphism $C^\infty(h):C^\infty(\RR^n) \rightarrow C^\infty(\RR^m)$ such that $\eta_{\RR^n}C^\infty(h) = h\eta_{\RR^m}$.  Hence, in view of Proposition \ref{thm:free_cinfty} and \ref{par:free_cinfty_on_fin_set} we deduce that $C^\infty(h) = (h\eta_{\RR^m})^\#$ sends each $g \in C^\infty(\RR^n)$ to
$$C^\infty(h)(g) = (h\eta_{\RR^m})^\#(g) = \Phi_g(\eta(h(e_1)),...,\eta(h(e_n))) = g \circ \langle \eta(h(e_1)),...,\eta(h(e_n))\rangle\;.$$
Letting $(h_{ij})$ be the matrix representation of $h$, we know that for each $j = 1,...,n$, 
\[h(e_j) = \sum_{i=1}^m h_{ij}e'_i\]
 where $e'_1,...,e'_m$ are the standard basis vectors for $\RR^m$, so by linearity 
 \[\eta(h(e_j)) = \sum_{i=1}^m h_{ij}\pi_i = \pi_j \circ h^*.\]
 Hence $C^\infty(h)(g) = g \circ \langle \pi_1 \circ h^*,...,\pi_n \circ h^*\rangle = g \circ h^*$.
\end{proof}

We now employ a characterization of algebra modalities in \cite{blo} to show that $\SSS^\infty$ carries the structure of an algebra modality (Definition \ref{def:alg_mod}).  Given a symmetric monoidal category $\C$, we shall denote by $\CMon(\C)$ the category of commutative monoids in $\C$.  If the forgetful functor $\CMon(\C) \rightarrow \C$ has a left adjoint, then we denote the induced monad on $\C$ by by $\Sym$ and call it the \textbf{symmetric algebra monad}, generalizing Example \ref{exa:sym_on_rvec}, and we say that \textit{the symmetric algebra monad exists}.

\begin{proposition}\label{thm:alg_mod}
Let $\C$ be a symmetric monoidal category $\C$ with reflexive coequalizers that are preserved by $\otimes$ in each variable, and assume that the symmetric algebra monad $\Sym$ on $\C$ exists.  The following are in bijective correspondence:
\begin{enumerate}
\item[(1)] algebra modalities $(\SSS,\mathsf{m},\mathsf{u})$ on $\C$;
\item[(2)] pairs $(\SSS,\lambda)$ consisting of a monad $\SSS$ on $\C$ and a monad morphism $\lambda:\Sym \rightarrow \SSS$;
\item[(3)] pairs $(\SSS,M)$ consisting of a monad $\SSS$ on $\C$ and a functor $M:\C^\SSS \rightarrow \CMon(\C)$ that commutes with the forgetful functors valued in $\C$.
\end{enumerate}
\end{proposition}
\begin{proof}
We briefly sketch the correspondences; the verifications are straightforward, and the existence of a bijection between (1) and (2) is asserted in \textnormal{\cite[Proposition 4.2]{blo}}, although with unnecessary blanket assumptions of additivity and finite biproducts.

Given $(\SSS,\mathsf{m},\mathsf{u})$ as in (1), with $\SSS = (S,\mu,\eta)$, the associated monad morphism $\lambda$ is obtained by defining $\lambda_C:\Sym(C) \rightarrow SC$ as the unique \textit{monoid} morphism such that $\eta^\Sym_C \lambda_C = \eta_C$, where $\eta^\Sym:1 \Rightarrow \Sym$ is the unit.

Given a monad $\SSS$ on $\C$, \cite[Proposition A.26]{AdamekRosickyVitale} yields a bijection between monad morphisms $\lambda:\Sym \rightarrow \SSS$ and functors $M:\C^\SSS \rightarrow \C^\Sym$ that commute with the forgetful functors to $\C$.  But the above hypotheses entail that the forgetful functor $V:\CMon(\C) \rightarrow \C$ is a right adjoint and creates reflexive coequalizers, so by the well-known Crude Monadicity Theorem (in the form given in \cite[Theorem 2.3.3.8]{Lu:PhD}) we deduce that $V$ is strictly monadic.  Hence $\CMon(\C) \cong \C^\Sym$ and the bijection between (2) and (3) is obtained.

Any functor $M$ as in (3) endows each free $\SSS$-algebra $SC$ with the structure of a commutative monoid in $\C$, which we may write as $(SC,\mathsf{m}_C,\mathsf{u}_C)$, and we thus obtain an algebra modality $(\SSS,\mathsf{m},\mathsf{u})$.
\end{proof}

\begin{corollary}\label{thm:free_cinfty_mod}
The free $C^\infty$-ring monad $\SSS^\infty$ on $\rVec$ carries the structure of an algebra modality $(\SSS^\infty,\mathsf{m},\mathsf{u})$.
\end{corollary}
\begin{proof}
$\CMon(\rVec) = \Alg{\RR}$, so this follows from Proposition \ref{thm:alg_mod} in view of \eqref{eq:alg_functors}.
\end{proof}

\begin{remark}\label{rem:sinfty_alg_mod_mnd_mor}
We call the algebra modality $(\SSS^\infty,\mathsf{m},\mathsf{u})$ the \textbf{free $C^\infty$-ring modality}.  For each real vector space $V$, $(S^\infty(V), \mathsf{m}_V,\mathsf{u}_V)$ is the $\RR$-algebra underlying the free $C^\infty$-ring on $V$.  In view of the proof of Proposition \ref{thm:alg_mod}, the corresponding monad morphism $\lambda:\Sym \rightarrow \SSS^\infty$ consists of mappings
$$\lambda_V:\Sym(V) \longrightarrow S^\infty(V)\;\;\;\;\;\;\;\;(V \in \rVec)$$
each characterized as the unique $\RR$-algebra homomomorphism with $\eta^\Sym_C \lambda_C = \eta_C$, where $\eta^\Sym:1 \Rightarrow \Sym$ and $\eta:1 \Rightarrow S^\infty$ denote the units.  In the case where $V = \RR^n$, we may identify $\Sym(\RR^n)$ with the polynomial $\RR$-algebra $\RR[x_1, \ldots, x_n]$, and $\lambda_{\RR^n}$ is simply the inclusion
	\[ \lambda_{\RR^n}:\Sym(\RR^n) = \RR[x_1, \ldots, x_n] \hookrightarrow \cInf(\RR^n)\;.\]
Indeed, the latter is an $\RR$-algebra homomorphism that sends the generators $x_i$ to the generators $\pi_i$ $(i=1,...,n)$.
\end{remark}

\section{Differential structure}\label{sec:diff}

Our goal in this section is to give codifferential structure for the free $C^\infty$-ring modality $\SSS^\infty$ (Corollary \ref{thm:free_cinfty_mod}).  Note that this was also done in the original differential categories paper \cite[\S 3]{differentialCats}, but for reasons explained in Section \ref{sec:intro} we will instead employ a different approach: we will exploit the fact that $\SSS^\infty$ is a finitary monad, in order to obtain its differential structure by left Kan-extension from structure on the finite-dimensional spaces, which we will describe explicitly.  This new approach will later enable us to also endow $\SSS^\infty$ with integral structure in Section \ref{sec:anti}.  Moreover, we believe that it is helpful to have multiple viewpoints on this key example.  

To demonstrate codifferential structure for $\SSS^\infty$, we will use the following theorem from \cite{blo}:

\begin{theorem}\label{thm:blo} \textnormal{\cite[6.1]{blo}} Suppose that $\C$ is an additive symmetric monoidal category with reflexive coequalizers that are preserved by the tensor product in each variable, and suppose that the symmetric algebra monad $\Sym$ on $\C$ exists.  Then to equip $\C$ with the structure of a codifferential category (in the sense of \cite{differentialCats}) is, equivalently, to equip $\C$ with
\begin{itemize}
	\item a monad $\mathsf{S} = (S,\eta,\mu)$;
	\item a monad morphism $\lambda: \Sym \to \mathsf{S}$;
	\item a natural transformation $\dt: SC \to SC \otimes C$ $\;(C \in \C)$;
\end{itemize}
such that
\begin{enumerate}[(a)]
	\item for each object $C$ of $\C$, 
	\[
	\xymatrix{\Sym(C) \ar[r]^{\lambda_{C}} \ar[d]_{\dt^{\Sym}_{C}} & S(C) \ar[d]^{\dt_{C}} \\ \Sym(C) \otimes V \ar[r]_{\lambda_{C} \otimes 1} & S(C) \otimes C}
	\]
	commutes (where $\dt^{\Sym}$ is the canonical deriving transformation on $\Sym$);
	\item the chain rule axiom of Definition \ref{defn:codiff} holds for $\dt$.
\end{enumerate}
\end{theorem}

It is important to note that this theorem gives codifferential structure in the original sense \cite{differentialCats}, not in the sense used in \cite{integralAndCalcCats}.  In particular, the above theorem gives codifferential structure satisfying the first four axioms of Definition \ref{defn:codiff}, but not necessarily the last axiom (interchange).  Hence, we will use the following corollary of this result:

\begin{corollary}\label{cor:blo}
To give a codifferential structure in the sense used in \cite{integralAndCalcCats} is equivalently to give structure as in Theorem \ref{thm:blo} such that the transformation $\dt$ also satisfies the interchange axiom \textnormal{\textbf{[d.5]}}: $\dt(\dt \otimes 1) = \dt (\dt \otimes 1)(1 \otimes \sigma)$.  
\end{corollary}

In \ref{rem:sinfty_alg_mod_mnd_mor} we have already equipped $\SSS^\infty$ with a monad morphism $\lambda:\Sym \rightarrow \SSS^\infty$.  We will define the deriving transformation first for the finitely presentable objects, i.e., the finite-dimensional vector spaces $\RR^n$, and then we will use Proposition \ref{prop:equiv_fin} and Example \ref{exa:rvec2} both to extend this definition to arbitrary vector spaces and to facilitate the checking of the required axioms for a deriving transformation.

\begin{definition}\label{defn:dFlat}
For each $n \in \mathbb{N}$, define $\dt_{\RR^n}^{\flat}: \cInf(\RR^n) \to \cInf(\RR^n) \otimes \RR^n$ by
	\[\dt_{\RR^n}^{\flat}(f) :=\sum_{i=1}^n \frac{\partial f}{\partial x_i} \otimes e_i\]
where $e_i$ denotes the $i$-th standard basis vector for $\RR^n$.
\end{definition}

\begin{remark}\label{rem:1forms}
Note that $C^\infty(\RR^n) \otimes \RR^n \cong (C^\infty(\RR^n))^n$ is a free, finitely-generated $C^\infty(\RR^n)$-module of rank $n$ and hence may be identified with the $C^\infty(\RR^n)$-module of smooth 1-forms on $\RR^n$, whereupon the basis elements $1 \otimes e_i$ of this free module $C^\infty(\RR^n) \otimes \RR^n$ are identified with the basic 1-forms $\mathsf{d}x_i$ $(i = 1,...,n)$ on $\RR^n$ (noting that then $\mathsf{d}^\flat_{\RR^n}(x_i) = \mathsf{d}x_i$ if one writes $x_i:\RR^n \rightarrow \RR$ for the $i$-th projection).

In particular, each element $\omega \in C^\infty(\RR^n) \otimes \RR^n$ can be expressed uniquely as 
\[\omega = \sum_{i=1}^n f_i \otimes e_i = \sum_{i=1}^n f_i \mathsf{d}x_i\] 
for smooth functions $f_i:\RR^n \rightarrow \RR$.  Since $(C^\infty(\RR^n))^n \cong C^\infty(\RR^n,\RR^n)$, each such 1-form $\omega \in C^\infty(\RR^n) \otimes \RR^n$ corresponds to a smooth vector field $F = \langle f_1,...,f_n\rangle:\RR^n \rightarrow \RR^n$ on $\RR^n$.

Given $f \in C^\infty(\RR^n)$, the 1-form $\mathsf{d}^\flat_{\RR^n}(f)$ defined in Definition \ref{defn:dFlat} is the usual differential of $f$ (also known as the exterior derivative of the 0-form $f$), whose corresponding vector field is the \textbf{gradient} of $f$
$$\nabla f = \left\langle \frac{\partial f}{\partial x_1},...,\frac{\partial f}{\partial x_n}\right\rangle\;:\;\RR^n \longrightarrow \RR^n.$$
\end{remark}

\begin{lemma}\label{lemma:dNatural}
The maps $\dt_{\RR^n}^{\flat}$ in Definition \ref{defn:dFlat} constitute a natural transformation $$\dt^\flat\;:\;C^\infty(-) \Longrightarrow C^\infty(-)\otimes(-)\;:\;\Lin_\RR \longrightarrow \rVec\;.$$
\end{lemma}
\begin{proof}
For this, we need to show that for any linear map $h: \RR^n \to \RR^m$,
\[
	\xymatrix{\cInf(\RR^n) \ar[r]^{\dt^{\flat}_{\RR^n}} \ar[d]_{\cInf(h)} & \cInf(\RR^n) \otimes \RR^n \ar[d]^{\cInf(h) \otimes h} \\ \cInf(\RR^m) \ar[r]_{\dt^{\flat}_{\RR^m}} & \cInf(\RR^m) \otimes \RR^m}
\]
commutes. Let $(h_{ij})$ be the matrix representation of $h$, let $h^*$ denote the adjoint (or transpose) of $h$, and let $(e_i)_{i=1}^m$ and $(e_j)_{j=1}^n$ denote the standard bases of $\RR^m$ and $\RR^n$, respectively.    Then for $f \in \cInf(\RR^n)$,
\begin{eqnarray*}
&   & \dt^{\flat}_{\RR^m} (\cInf(h)(f)) \\
& = & \dt^{\flat}_{\RR^m} (f \circ h^*)\;\;\;\;\;\;\text{(by Proposition \ref{thm:charn_cinfty})} \\
& = & \sum_{i=1}^m \frac{\di (f \circ h^*)}{\di x_i} \otimes e_i \\
& = & \sum_{i=1}^m \left[ \sum_{j=1}^n \left( \frac{\di f}{\di x_j} \circ h^* \right) \frac{\di h^*_j}{\di x_i} \right] \otimes e_i \mbox{ (by the chain rule)} \\
& = & \sum_{i=1}^m \left[ \sum_{j=1}^n \left( \frac{\di f}{\di x_j} \circ h^* \right) h_{ij} \right] \otimes e_i \mbox{ (by the matrix representation of $h^*$)} \\
& = & \sum_{j=1}^n \left( \frac{\di f}{\di x_j} \circ h^* \right)  \otimes \left( \sum_{i=1}^m h_{ij} e_i \right) \mbox{ (by bilinearity of $\otimes$) } \\
& = & \sum_{j=1}^n \left( \frac{\di f}{\di x_j} \circ h^* \right)  \otimes h(e_j) \mbox{ (by the matrix representation of $h$) } \\
& = & (\cInf(h) \otimes h)\left( \sum_{j=1}^n \frac{ \di f}{\di x_j} \otimes e_j \right) \\
& = & (\cInf(h) \otimes h)\left( \dt^{\flat}_{\RR^n}(f) \right)
\end{eqnarray*}
as required.  
\end{proof}

\begin{lemma}\label{thm:lem_sinfty_tensor_id}\emptybox
\begin{enumerate}
\item The functor $S^\infty(-) \otimes (-):\rVec \rightarrow \rVec$ is finitary.
\item The restriction of $S^\infty(-)\otimes (-)$ to $\Lin_\RR$ is precisely $C^\infty(-)\otimes (-)$.
\item $S^\infty(-) \otimes (-)$ is a left Kan extension of $C^\infty(-) \otimes (-):\Lin_\RR \rightarrow \rVec$ along the inclusion $\iota:\Lin_\RR \hookrightarrow \rVec$.
\end{enumerate}
\end{lemma}
\begin{proof}
(2) follows from the fact that the restriction $S^\infty\iota$ of $S^\infty$ to $\Lin_\RR$ is precisely $C^\infty$ (on the nose, by \ref{rem:cinf_sinf}).  Since $S^\infty$ and $1_{\rVec}$ are finitary, we deduce by Proposition \ref{thm:tens_pr_fin} that (1) holds, and (3) then follows, by Example \ref{exa:rvec2} and Proposition \ref{prop:equiv_fin}.
\end{proof}

\begin{definition}\label{def:smooth_der_transf}
Using Lemma \ref{thm:lem_sinfty_tensor_id}, we define
$$\dt\;:\;S^\infty(-) \Longrightarrow S^\infty(-) \otimes (-)\;:\;\rVec \rightarrow \rVec$$
to be the natural transformation
$$\dt = \Lan_\iota(\dt^\flat) \;:\;\Lan_\iota(C^\infty) \Longrightarrow \Lan_\iota(C^\infty(-)\otimes (-))$$
corresponding to $\dt^\flat$ under the equivalence $\Lan_\iota:[\Lin_\RR,\rVec] \rightarrow \Fin(\rVec,\rVec)$ of Example \ref{exa:rvec2} and Proposition \ref{prop:equiv_fin}. In view of Lemma \ref{thm:lem_sinfty_tensor_id}.(2), we note that $\iota^*(\dt) = \dt^\flat$ in the notation of Example \ref{exa:rvec2}, i.e.,
$$\dt_{\RR^n} = \dt^\flat_{\RR^n}\;\;\;\;\;\;(n \in \NN).$$
\end{definition}

\begin{lemma}\label{lemma:conditionA}
$\dt$ satisfies (a) of Theorem \ref{thm:blo} for the objects $C = \mathbb{R}^n$ $(n \in \NN)$; that is,
\begin{equation}\label{eq:diag_condA}
	\xymatrix{\RR[x_1 \ldots x_n] \ar[r]^{\lambda_{\RR^n}} \ar[d]_{\dt^{\Sym}_{\RR^n}} & \cInf(\RR^n) \ar[d]^-{\dt_{\RR^n}} \\ \RR[x_1 \ldots x_n] \otimes \RR^n \ar[r]_-{\lambda_{\RR^n} \otimes 1} & \cInf(\RR^n) \otimes \RR^n}
\end{equation}
commutes.
\end{lemma}
\begin{proof}
This is immediate since by \ref{rem:sinfty_alg_mod_mnd_mor}, $\lambda_{\RR^n}$ is the inclusion, and when the formula for $\dt_{\RR^n} = \dt^{\flat}_{\RR^n}$ is applied to a polynomial, we recover the formula for $\dt^{\Sym}_{\RR^n}$ (see Example \ref{ex:rVecDiff}).
\end{proof}

\begin{corollary}\label{thm:cor_cond_a}
$\dt$ satisfies (a) of Theorem \ref{thm:blo}.
\end{corollary}
\begin{proof}
\ref{thm:blo}.(a) requires that
$$\lambda\dt = \dt^\Sym(\lambda \otimes 1)\;:\;\Sym \Longrightarrow S^\infty(-) \otimes (-)\;:\;\rVec \longrightarrow \rVec\;.$$
The components at $\RR^n$ of these two natural transformations $\lambda\dt$ and $\dt^\Sym(\lambda \otimes 1)$ are precisely the two composites in \eqref{eq:diag_condA}, so since $\Sym$ and $S^\infty(-) \otimes (-)$ are finitary functors and $\iota^*:\Fin(\rVec,\rVec) \rightarrow [\Lin_\RR,\rVec]$ is an equivalence (Example \ref{exa:rvec2}), the result follows.
\end{proof}

\begin{lemma}\label{lemma:chain}
$\dt$ satisfies the chain rule for the objects $\mathbb{R}^n$; that is, the following diagram commutes:
\[
	\xymatrix{ S^\infty(\cInf(\RR^n)) \ar[rr]^{\mu} \ar[d]_-{\dt} & & \cInf(\RR^n) \ar[d]^-{\dt} \\ S^\infty(\cInf(\RR^n)) \otimes \cInf(\RR^n) \ar[r]_-{\mu \otimes \dt} & \cInf(\RR^n) \otimes \cInf(\RR^n) \otimes \RR^n \ar[r]_-{m \otimes 1} & \cInf(\RR^n) \otimes \RR^n}
\]
\end{lemma}
\begin{proof}
By \ref{rem:cinf_sinf}, we know that the maps
$$S^\infty(\phi)\;:\;S^\infty(\RR^m) = C^\infty(\RR^m) \longrightarrow S^\infty(C^\infty(\RR^n))\;,\;\;\;\;\;\;(\RR^m,\phi) \in \Lin_\RR\slash C^\infty(\RR^n)\;,$$
present $S^\infty(C^\infty(\RR^n))$ as a colimit in $\rVec$.  Hence, to check the commutativity of the diagram above, it suffices to let $\phi: \RR^m \to \cInf(\RR^n)$ be a linear map and check that the diagram commutes when pre-composed by $S^\infty(\phi)$.  So, we will first consider the upper-right composite:
\begin{equation}\label{eq:comp} 
\xymatrix{ \cInf(\RR^m) \ar[r]^(.4){S^\infty(\phi)} &  S^\infty(\cInf(\RR^n)) \ar[r]^{\mu} & \cInf(\RR^n) \ar[r]^{\dt} & \cInf(\RR^n) \otimes \RR^n}
\end{equation}
Since $C^\infty(\RR^m) = S^\infty(\RR^m)$ is the free $\SSS^\infty$-algebra on the vector space $\RR^m$, we deduce by generalities on Eilenberg-Moore categories that the composite $S^\infty(\phi)\mu$ of the first two morphisms in \eqref{eq:comp} is the unique $\SSS^\infty$-algebra homomorphism $\phi^\#:C^\infty(\RR^m) \rightarrow C^\infty(\RR^n)$ such that $\eta \phi^\# = \phi$.  Hence, by Proposition \ref{thm:free_cinfty} and \ref{par:free_cinfty_on_fin_set} we deduce that $\phi^\# = S^\infty(\phi)\mu$ is given by
$$\phi^\#(g) = \Phi_g(\phi(e_1),...,\phi(e_m)) = g \circ \langle \phi(e_1),...,\phi(e_m)\rangle \;\;\;\;\;(g \in C^\infty(\RR^m))\;.$$
Hence, letting $\alpha_i = \phi(e_i)$ for each $i = 1,...,m$ and letting $\alpha = \langle \alpha_1,...,\alpha_m\rangle:\RR^n \rightarrow \RR^m$, we know that $\phi^\#(g) = g \circ \alpha$.  Therefore
	\[ (S^\infty(\phi) \mu \dt)(g) = (\phi^\#\dt)(g) = \dt(g \circ \alpha) = \sum_{i=1}^n \frac{\di (g \circ \alpha)}{\di x_i} \otimes e_i \ (\dagger). \]

We now calculate the lower-left composite when pre-composed by $S^\infty(\phi)$.  By the naturality of $\dt$, $S^\infty(\phi) \dt = \dt (S^\infty(\phi) \otimes \phi)$.  Also, $\phi^\# = S^\infty(\phi) \mu$, so we are considering the composite
	\[ (\dt)(\phi^\# \otimes \phi)(1 \otimes \dt)(m \otimes 1): \cInf(\RR^m) \to \cInf(\RR^n) \otimes \RR^n. \]
We now calculate the result of applying this composite to each $g \in \cInf(\RR^m)$.  
\begin{itemize}
	\item Applying $\dt$ to $g$ gives 
		\[ \sum_{j=1}^m \frac{\di g}{\di x_j} \otimes e_j. \]
	\item As above, $\phi^\#(g) = g \circ \alpha$, so applying $\phi^\# \otimes \phi$ to this gives
		\[ \sum_{j=1}^m \left( \frac{\di g}{\di x_j} \circ \alpha \right) \otimes \alpha_j. \]
	\item Applying $1 \otimes \dt$ to this gives
		\[ \sum_{j=1}^m \left( \frac{\di g}{\di x_j} \circ \alpha \right) \otimes \sum_{i=1}^n \frac{\di \alpha_j}{\di x_i} \otimes e_i = \sum_{j=1}^m \sum_{i=1}^n \left( \frac{\di g}{\di x_j} \circ \alpha \right) \otimes \frac{\di \alpha_j}{\di x_i} \otimes e_i \]
	\item And then applying $\m \otimes 1$ gives 
		\[ \sum_{j=1}^m \sum_{i=1}^n \left( \frac{\di g}{\di x_j} \circ \alpha \right) \frac{\di \alpha_j}{\di x_i} \otimes e_i = \sum_{i=1}^n \left[ \sum_{j=1}^m \left( \frac{\di g}{\di x_j} \circ \alpha \right) \frac{\di \alpha_j}{\di x_i} \right] \otimes e_i \]
\end{itemize}
However, by the ordinary chain rule for smooth functions, this last expression equals $\dagger$, as required.  
\end{proof}

\begin{corollary}\label{thm:cor_chain_rule}
$\dt$ satisfies the chain rule ([d.4] in Definition \ref{defn:codiff}).
\end{corollary}
\begin{proof}
This follows from Lemma \ref{lemma:chain} by an argument as in Corollary \ref{thm:cor_cond_a}, using Example \ref{exa:rvec2}, since the chain rule is the equality of a parallel pair of natural transformations $S^\infty(S^\infty(-)) \Rightarrow S^\infty(-)\otimes(-)$ between finitary endofunctors on $\rVec$.
\end{proof}

\begin{lemma}\label{lemma:interchange}
$\dt$ satisfies the interchange rule for the objects $\mathbb{R}^n$; that is, the following diagram commutes:
\[
	\xymatrix{ \cInfRn \ar[rr]^{\dt} \ar[d]_{\dt} & & \cInfRn\otimes \RR^n \ar[d]^{\dt \otimes 1} \\ \cInfRn \otimes \RR^n \ar[r]_{\dt \otimes 1} & \cInf(\RR^n) \otimes \RR^n \otimes \RR^n \ar[r]_{1 \otimes \sigma } & \cInfRn \otimes \RR^n \otimes \RR^n}
\]
\end{lemma}
\begin{proof}
Given some $f \in \cInfRn$, by definition,
	\[\dt_{\RR^n}(f) =\sum_{i=1}^n \frac{\partial f}{\partial x_i} \otimes e_i\;.\] 
Then applying $\dt \otimes 1$ to this, we get
	\[ \sum_{i=1}^n \sum_{j=1}^n \frac{\partial^2 f}{\partial x_j \partial x_i} \otimes e_j  \otimes e_i, (\dagger) \]
and applying $1 \otimes \sigma$ to this gives
	\[ \sum_{i=1}^n \sum_{j=1}^n \frac{\partial^2 f}{\partial x_j \partial x_i} \otimes e_i \otimes e_j = \sum_{i=1}^n \sum_{j=1}^n \frac{\partial^2 f}{\partial x_i \partial x_j} \otimes e_j \otimes e_i. \]
But this is equal to $\dagger$, by the symmetry of mixed partial derivatives.  
\end{proof}

\begin{corollary}\label{thm:interchange}
$\dt$ satisfies the interchange rule ([d.5] in Definition \ref{defn:codiff}).
\end{corollary}
\begin{proof}
By an argument as in Lemma \ref{thm:lem_sinfty_tensor_id}
, $S^\infty(-) \otimes (-) \otimes (-)$ is a finitary endofunctor on $\rVec$ whose restriction to $\Lin_\RR$ is precisely $C^\infty(-) \otimes (-) \otimes (-)$.  The result now follows from Lemma \ref{lemma:interchange} by an argument as in Corollaries \ref{thm:cor_cond_a} and \ref{thm:cor_chain_rule}.
\end{proof}

\begin{theorem}
$\rVec$ has the structure of a codifferential category when equipped with the free $C^\infty$-ring monad $\SSS^\infty$.
\end{theorem}
\begin{proof}
In view of Corollary \ref{cor:blo}, this follows from Remark \ref{rem:sinfty_alg_mod_mnd_mor}, Definition \ref{def:smooth_der_transf}, and Corollaries \ref{thm:cor_cond_a}, \ref{thm:cor_chain_rule}, and \ref{thm:interchange}.
\end{proof}

\begin{remark}\label{remark:storage}
Recall that an algebra modality is said to have the \emph{storage} or \emph{Seely} isomorphisms if certain canonical morphisms $S(X) \otimes S(Y) \rightarrow S(X \times Y)$ and $k \rightarrow S(1)$ are isomorphisms\footnote{See \cite[Definition 7.1]{blute2018differential}, where the dual notion is defined.}.  The algebra modality considered here does not have this property, as even for $X = Y = \RR$ the canonical map $\cInf(\RR) \otimes \cInf(\RR) \rightarrow \cInf(\RR \times \RR) = \cInf(\RR^2)$ is not surjective, as its image does not contain\footnote{Example from math.stackexchange.com page \#2244402.} the function $e^{xy}$.  As noted in the introduction, this is then a crucial example of a (co)differential category that does not have the Seely isomorphisms.  
\end{remark}

\setcounter{subsection}{0}
\subsection{$\mathsf{S}^\infty$-Derivations}\label{derivationsec}

The algebraic notions of derivation and K{\"a}hler differential generalize to the setting of codifferential categories, through the notion of $\mathsf{S}$-derivation introduced by Blute, Lucyshyn-Wright, and O'Neill \cite{blo}.  Here, we prove that the $\mathsf{S}^\infty$-derivations for the codifferential category given by the free $C^\infty$-ring modality correspond precisely to derivations relative to the Fermat theory of smooth functions in the sense of Dubuc and Kock  \cite{dubucKock1Form}.  This thus demonstrates the importance of the general notion of $\mathsf{S}$-derivation defined in \cite{blo}, and provides a key link between differential categories and previous work in categorical differential geometry.  

For an algebra modality $(\SSS,\mathsf{m},\mathsf{u})$ on a symmetric monoidal category, every $\mathsf{S}$-algebra comes equipped with a commutative monoid structure \cite[Theorem 2.12]{blo}. Indeed, if $(A, \nu)$ is an $\mathsf{S}$-algebra for the monad $\SSS = (S,\mu,\eta)$ (where we recall that $\nu: S(A) \to A$ is a morphism satisfying certain equations involving $\eta$ and $\mu$), we define a commutative monoid structure on $A$ with multiplication $\mathsf{m}^\nu: A \otimes A \to A$ and unit $\mathsf{u}^\nu: k \to A$ defined respectively as follows: 
\[\mathsf{m}^\nu := (\eta_A \otimes \eta_A)\mathsf{m}_A \nu \quad \quad \quad \mathsf{u}^\nu := \mathsf{u}_A \nu \] 
Notice that for free $\SSS$-algebras $(S(C), \mu_C)$, $\mathsf{m}^{\mu_C}= \mathsf{m}_C$ and $\mathsf{u}^{\mu_C} = \mathsf{u}_C$. We may now also consider modules over an $\SSS$-algebra $(A, \nu)$, or rather modules over the commutative monoid $(A, \mathsf{m}^\nu, \mathsf{u}^\nu)$, which we recall are pairs $(M, \alpha)$ consisting of an object $M$ and a morphism $\alpha: A \otimes M \to M$ satisfying the standard coherences. 

\begin{definition_sub}\label{defn:Sderivation}
Let $\C$ be a codifferential category with algebra modality $(\SSS,\mathsf{m},\mathsf{u})$ and deriving transformation $\mathsf{d}$. Given an $\SSS$-algebra $(A, \nu)$ and an $(A, \mathsf{m}^\nu, \mathsf{u}^\nu)$-module $(M, \alpha)$, an \textbf{$\SSS$-derivation} \cite[Definition 4.12]{blo} is a morphism $\partial: A \to M$ such that the following diagram commutes:
   \[  \xymatrixcolsep{5pc}\xymatrix{ S(A) \ar[d]_-{\mathsf{d}_A} \ar[rr]^-{\nu} && A \ar[d]^-{\partial} \\
   S(A) \otimes A \ar[r]_-{\nu \otimes \partial} & A \otimes M \ar[r]_-{\alpha} & M
  } \]
\end{definition_sub}

The canonical example of an $\SSS$-derivation is the deriving transformation \cite[Theorem 4.13]{blo}. Indeed for each object $C$, $\mathsf{d}_C$ is an $\SSS$-derivation on the $\SSS$-algebra $(S(C), \mu_C)$ valued in the module $(S(C) \otimes C, \mathsf{m}_C \otimes 1_C)$. $\SSS$-derivations are the appropriate generalization of the classical notion of derivation, as every $\SSS$-derivation is a derivation in the classical sense. The key difference is that classical derivations are axiomatized by the Leibniz rule, while $\SSS$-derivations are axiomatized by the chain rule. In the special case of the symmetric algebra monad, under the assumptions of \ref{thm:blo}, $\mathsf{Sym}$-derivations correspond precisely to derivations in the classical sense \cite[Remark 5.8]{blo}.
But what do $\mathsf{S}^\infty$-derivations correspond to? For this, we turn to Dubuc and Kock's generalized notion of derivation for Fermat theories. While we will not review Fermat theories in general (we invite the curious reader to learn about them in \cite{dubucKock1Form}), we will instead consider Dubuc and Kock's generalized derivations for the Fermat theory of smooth functions, which are explicitly described by Joyce in \cite{joyce2011introduction}.  

\begin{definition_sub}\label{defn:Cderivation}
Given a $\cInfRing$ $(A, \Phi)$ and an $A$-module $M$ (that is, $M$ is a module over the underlying ring structure of $A$), a \textbf{$C^\infty$-derivation} \cite{dubucKock1Form,joyce2011introduction} is a map $\mathsf{D}: A \to M$ such that for each smooth function $f: \mathbb{R}^n \to \mathbb{R}$, the following equality holds: 
\[\mathsf{D}\left(\Phi_f(a_1, \hdots, a_n)\right) = \sum\limits^n_{i = 1} \Phi_{\frac{\partial f}{\partial x_i}}(a_1, \hdots, a_n) \cdot \mathsf{D}(a_i)\]
for all $a_1, \hdots, a_n \in A$ (and where $\cdot$ is the $A$-module action). 
\end{definition_sub}

To see why $C^\infty$-derivations are precisely the same thing as $\mathsf{S}^\infty$-derivations, we will first take a look at equivalent definitions for each of these generalized derivations. In the presence of biproducts, arbitrary $\SSS$-derivations in a codifferential category can equivalently be described as certain $\SSS$-algebra morphisms \cite[Definition 4.7]{blo}. This generalizes a well-known result on derivations in commutative algebra.

\begin{theorem_sub}\label{thm:SderSalg}\textnormal{\cite[Theorem 4.1, Proposition 4.11]{blo}} 
Let $\C$ be a codifferential category with algebra modality $(\SSS,\mathsf{m},\mathsf{u})$ and deriving transformation $\mathsf{d}$.  Let $(A, \nu)$ be an $\SSS$-algebra and $(M, \alpha)$ an $(A, \mathsf{m}^\nu, \mathsf{u}^\nu)$-module, and suppose that $\C$ has finite biproducts $\oplus$. Then the pair $(A \oplus M, \beta)$ is an $\SSS$-algebra where $\beta:S(A \oplus M) \to A \oplus M$ is defined as
\[\beta := \langle S(\pi_1) \nu, \mathsf{d}_{A \oplus M} (S(\pi_1) \otimes \pi_2)(\nu \otimes 1) \alpha \rangle\]
where $\pi_1:A \oplus M \rightarrow A$ and $\pi_2:A \oplus M \rightarrow M$ are the projections.
Furthermore, a morphism $\partial: A \to M$ is an $\SSS$-derivation if and only if $\langle 1_A, \partial \rangle: (A, \nu) \to (A \oplus M, \beta)$ is an $\SSS$-algebra morphism. 
\end{theorem_sub}

More general statements regarding the equivalence between $\SSS$-derivations and $\SSS$-algebra morphisms can be found in \cite{blo}. In the case of the free $\cInfRing$ monad, $\mathsf{S}^\infty$-algebra morphisms correspond precisely to $\cInfRing$ morphisms. Therefore to give an $\mathsf{S}^\infty$-derivation $\partial: A \to M$ amounts to giving a $\cInfRing$ morphism $\langle 1_A, \partial \rangle: A \to A \oplus M$, where $A \oplus M$ carries the $C^\infty$-ring structure corresponding to the $\SSS^\infty$-algebra structure $\beta$ in Theorem \ref{thm:SderSalg}.  This is similar to a result for algebras of Fermat theories that had been given earlier by Kock and Dubuc, stated here for $C^\infty$-rings and $C^\infty$-derivations:

 \begin{theorem_sub}\label{thm:CderCmorph}\textnormal{\cite[Proposition 2.2]{dubucKock1Form}} 
Let $(A, \Phi)$ be a $\cInfRing$ and $M$ an $A$-module. Then $(A \oplus M, \tilde \Phi)$ is a $\cInfRing$ where  for each smooth function $f: \mathbb{R}^n \to \mathbb{R}$, $\tilde \Phi_f: (A\oplus M)^n \to A \oplus M$ is defined as follows: 
\[\tilde \Phi_f \left((a_1, m_1), \hdots, (a_n, m_n) \right) = \left( \Phi_f(a_1, \hdots, a_n),  \sum\limits^n_{i=1} \Phi_{\frac{\partial f}{\partial x_i}}(a_1, \hdots, a_n) \cdot m_i \right) \]
for all $a_1, \hdots, a_n \in A$ (and where $\cdot$ is the $A$-module action). Furthermore, a map $\mathsf{D}: A \to M$ is a $C^\infty$-derivation if and only if $\langle 1_A, \mathsf{D} \rangle: (A, \Phi) \to (A \oplus M, \tilde \Phi)$ is a $\cInfRing$ morphism. 
\end{theorem_sub}

We note that \cite[Proposition 2.2]{dubucKock1Form} is in fact a more general statement than what is stated in Theorem \ref{thm:CderCmorph}. It is not difficult to see that the $\cInfRing$ $(A \oplus M, \tilde \Phi)$ from Theorem \ref{thm:CderCmorph} corresponds precisely to the $\mathsf{S}^\infty$-algebra $(A \oplus M, \beta)$ from Theorem \ref{thm:SderSalg}. Therefore since $\SSS^\infty$-algebra morphisms are equivalently described as $\cInfRing$ morphisms, it follows that $\mathsf{S}^\infty$-derivations are equivalently described as $C^\infty$-derivations: 

\begin{theorem_sub} For the codifferential category structure on $\rVec$ induced by the free $\cInfRing$ monad $\mathsf{S}^\infty$, the following are in bijective correspondence: 
\begin{enumerate}[{\em (i)}]
\item $\mathsf{S}^\infty$-derivations (Definition \ref{defn:Sderivation});
\item $C^\infty$-derivations (Definition \ref{defn:Cderivation}). 
\end{enumerate}
\end{theorem_sub} 

An immediate consequence of this theorem is that universal $\mathsf{S}^\infty$-derivations correspond to universal $C^\infty$-derivations. For arbitrary codifferential categories, universal $\SSS$-derivations $\partial: A \to \Omega_A$  \cite[Definition 4.14]{blo} are generalizations of K{\"a}hler differentials, where in particular, $\Omega_A$ is the generalization of the classical module of K{\"a}hler differentials of a commutative algebra. Similarly, universal derivations of Fermat theories \cite[Theorem 2.3]{dubucKock1Form} provide a simultaneous generalization of both K{\"a}hler differentials and smooth $1$-forms. Indeed, it is well known that for a smooth manifold $M$, the module of K{\"a}hler differentials (in the classical sense) of $C^\infty(M)$ is not, in general, the module of smooth $1$-forms of $M$. We can explain this phenomenon in the following way: the module of K{\"a}hler differentials is the universal $\mathsf{Sym}$-derivation, and not the universal $C^\infty$-derivation. For $C^\infty(M)$, the universal $C^\infty$-derivation (equivalently, the universal $S^\infty$-derivation) is in fact the module of smooth $1$-forms of $M$ \cite{dubucKock1Form}.  Looking back at arbitrary codifferential categories, this justifies the use of the more general $\SSS$-derivations to study de Rham cohomology of $\SSS$-algebras \cite{o2017smoothness}. 

\subsection{A quasi-codereliction}\label{codersec}

In a codifferential category where $S(C)$ admits a natural bialgebra structure, the differential structure can equivalently be axiomatized by a \textbf{codereliction} \cite{blute2018differential, differentialCats}, which is in particular an $\SSS$-derivation. We will see that although $\mathsf{S}^\infty$ does not have this structure, it is still possible to construct a sort of `quasi-codereliction' that satisfies identities that are similar to the axioms of a codereliction.  

\begin{definition_sub} An algebra modality $\algMod$ on an additive symmetric monoidal category $(\C, \otimes, k, \sigma)$ is said to be an \textbf{additive bialgebra modality} \cite{blute2018differential} if it comes equipped with a natural transformation $\Delta$ with components $\Delta_C: S(C) \to S(C) \otimes S(C)$ and a natural transformation $\mathsf{e}$ with components $\mathsf{e}_C: S(C) \to k$, such that
\begin{itemize}
\item for each $C$ in $\C$, $(SC,\mathsf{m}_C, \mathsf{u}_C, \Delta_C, \mathsf{e}_C)$ is a commutative and cocommutative bimonoid (in the symmetric monoidal category $\C$);
\item the following equations are satisfied: 
\[ \eta \Delta = (\mathsf{u} \otimes \eta) + (\eta \otimes \mathsf{u}), \quad \quad \eta \mathsf{e} = 0;\]
\item for each pair of morphisms $f: A \to B$ and $g: A \to B$, the following equality holds: 
\[S(f+g) = \Delta_A \left( S(f) \otimes S(g) \right) \mathsf{m}_B; \]
\item for each zero morphism $0: A \to B$, the following equality holds: 
\[S(0)= \mathsf{e}_A \mathsf{u}_B.  \]
\end{itemize}
\end{definition_sub}

\begin{definition_sub} A \textbf{codereliction} \cite{blute2018differential, differentialCats} for an additive bialgebra modality is a natural transformation $\varepsilon: S \rightarrow 1_\C$ such that: 
\begin{enumerate}[{\bf [dc.1]}]
    \item \textbf{Constant rule}: $\mathsf{u}\varepsilon = 0$;
    \item \textbf{Leibniz/product rule}: $\m \varepsilon = (\mathsf{e} \otimes \varepsilon) + (\varepsilon \otimes \mathsf{e})$;
    \item \textbf{Derivative of a linear function}: $\eta \varepsilon = 1$;
    \item \textbf{Chain rule}: $\mu \Delta (1 \otimes \epsilon) = \Delta (\mu \otimes \epsilon)(1 \otimes \Delta)(\m \otimes \epsilon)$.  
\end{enumerate}
\end{definition_sub}

The intuition for coderelictions is best understood as evaluating derivatives at zero. For additive bialgebra modalities, there is a bijective correspondence between deriving transformations and coderelictions  \cite{blute2018differential}. Indeed, every codereliction induces a deriving transformation, defined by
  \[ \mathsf{d} := \Delta (1 \otimes \varepsilon)\;,   \]
and, conversely, every deriving transformation induces a codereliction:
\begin{equation}\label{coderdef}\begin{gathered} \varepsilon := \mathsf{d}(\mathsf{e} \otimes 1)\;. \end{gathered}\end{equation}
Therefore, note that the codereliction chain rule \textbf{[dc.4]} is then precisely the deriving transformation chain rule \textbf{[d.4]}. Post-composing both sides of the chain rule with $(\mathsf{e} \otimes 1)$, one obtains the following identity (called the alternative chain rule in \cite{blute2018differential})
$$\mu\varepsilon = \Delta(\mu \otimes\varepsilon)(\mathsf{e}\otimes\varepsilon)\;,$$
which we can rewrite equivalently as
\begin{equation}\label{altchain}\begin{gathered} \mu \varepsilon = \dt (\mu \otimes \varepsilon) (\mathsf{e} \otimes 1)\;. \end{gathered}\end{equation}
Now since $S(C)$ is a bialgebra, the morphism $\mathsf{e}_C \otimes 1_A: S(C) \otimes A \to A$ is a $(S(C), \mathsf{m}, \mathsf{u})$-module action for every object $A$. Therefore, we obtain the following observation: 

\begin{lemma_sub} For every object $C$, $\varepsilon_C: S(C) \to C$ is an $\SSS$-derivation on the free $\SSS$-algebra $(S(C), \mu_C)$ valued in the $(S(C), \mathsf{m}_C, \mathsf{u}_C)$-module $(C, \mathsf{e}_C \otimes 1_C)$.    
\end{lemma_sub}

In the presence of biproducts, additive bialgebra modalities are equivalently described as algebra modalities that have the Seely isomorphisms (see \cite{blute2018differential} for more details). Therefore, as mentioned in \ref{remark:storage}, the algebra modality $\mathsf{S}^\infty$ is not an additive bialgebra modality since it does not have the Seely isomorphisms; alternatively, one can directly argue that the vector spaces $C^\infty(\mathbb{R}^n)$ are not bialgebras since they do not have comultiplications. On the other hand to construct a codereliction from a deriving transformation as in \eqref{coderdef}, one only needs to have a counit, which $\mathsf{S}^\infty$ does have. Indeed, define the natural transformation $\mathsf{e}^\flat: C^\infty \Rightarrow \mathbb{R}$ by declaring that for each finite-dimensional vector space $\RR^n$, the map $\mathsf{e}^\flat_{\RR^n}:C^\infty(\mathbb{R}^n) \to \mathbb{R}$ is given by evaluation at zero:
\[\mathsf{e}^\flat_{\RR^n}(f)= f(\vec 0)\;.\]
One can check that $\mathsf{e}^\flat_{\RR^n}$ is also a $C^\infty$-ring morphism. Then define $\mathsf{e}:  S^\infty \Rightarrow \mathbb{R}$ as the image of $\mathsf{e}^\flat$ under the equivalence of Example \ref{exa:rvec2}, i.e., $\mathsf{e} = \Lan_\iota(\mathsf{e}^\flat)$ (recalling that $S^\infty = \Lan_\iota(C^\infty)$ and noting that the constant functor $\RR:\rVec \rightarrow \rVec$ is finitary and is a left Kan extension, along $\iota$, of the constant functor $\RR:\Lin_\RR \rightarrow \rVec$). We then define the natural transformation $\varepsilon: S^\infty \Rightarrow 1_{\rVec}$ in the same manner that a codereliction was defined in \eqref{coderdef}. Explicitly, $\varepsilon$ is defined component-wise as follows: 
   \[ \varepsilon_V :=  \xymatrixcolsep{5pc}\xymatrix{ S^\infty(V) \ar[r]^-{\mathsf{d}_V} & S^\infty(V) \otimes V \ar[r]^-{\mathsf{e}_V \otimes 1_V} & V 
  } \]
  In particular, $\varepsilon_{\RR^n}: C^\infty(\RR^n) \to \RR^n$ is the linear map that evaluates the derivative at zero: 
  \[ \varepsilon_{\RR^n}(f)= \sum_{i=1}^n \mathsf{e}^\flat_{\RR^n}\left( \frac{\partial f}{\partial x_i} \right) \otimes e_i = \sum_{i=1}^n \frac{\partial f}{\partial x_i} (\vec 0) e_i= \left( \frac{\partial f}{\partial x_1}(0), \hdots,  \frac{\partial f}{\partial x_n}(0) \right), \]
where the pure tensors in the second expression are taken in $\RR \otimes V = V$.  Now note that the first three codereliction axioms \textbf{[dc.1]}, \textbf{[dc.2]}, and \textbf{[dc.3]} involve the algebra modality structure and the counit $\mathsf{e}$, but not the comultiplication. As a consequence of the deriving transformation axioms, it follows that $\varepsilon$ satisfies \textbf{[dc.1]}, \textbf{[dc.2]}, and \textbf{[dc.3]}. The remaining axiom, the codereliction chain rule \textbf{[dc.4]}, involves the comultiplication, which we cannot express with $\mathsf{S}^\infty$. However by construction, $\varepsilon$ satisfies the alternative chain rule (\ref{altchain}), which in a sense replaces \textbf{[dc.4]}, and requires precisely that $\varepsilon$ be an $\mathsf{S}^\infty$-derivation. This makes $\varepsilon$ a sort of \textit{quasi}-codereliction for $\mathsf{S}^\infty$. We summarize this result as follows:  

\begin{proposition_sub} The natural transformation $\varepsilon: \mathsf{S}^\infty \Rightarrow 1_{\rVec}$ satisfies the following equalities:
\begin{enumerate}[{\em (i)}]
\item $\mathsf{u}\varepsilon = 0$;
\item $\m \varepsilon = (\mathsf{e} \otimes \varepsilon) + (\varepsilon \otimes \mathsf{e})$;
\item  $\eta \varepsilon = 1$;
\item $\mu \varepsilon = \dt (\mu \otimes \varepsilon) (\mathsf{e} \otimes 1)$.
\end{enumerate}
In particular, for every $\mathbb{R}$-vector space $V$, $\varepsilon_V: S^\infty(V) \to V$ is an $\mathsf{S}^\infty$-derivation (or equivalently a $C^\infty$-derivation) on the free $\mathsf{S}^\infty$-algebra $(S^\infty(V), \mu)$, valued in the $(S^\infty(V), \mathsf{m}_V, \mathsf{u}_V)$-module $(V, \mathsf{e}_V \otimes 1_V)$.    
\end{proposition_sub}

\section{Antiderivatives and integral structure} \label{sec:anti}

The goal of this section is to show that the codifferential category structure on $\rVec$ induced by the free $\cInfRing$ monad has antiderivatives, and that therefore we obtain a calculus category (and hence also an integral category). Explicitly, we wish to show that the natural transformation $\mathsf{K}: S^\infty \Rightarrow S^\infty$ (Definition \ref{def:K}) is a natural isomorphism.  However, the finitary functor $S^\infty:\rVec \rightarrow \rVec$ is a left Kan extension of its own restriction $C^\infty:\Lin_\RR \rightarrow \rVec$ (\ref{rem:cinf_sinf}).  Hence, in keeping with the strategy of Section \ref{sec:diff}, it suffices to show that the restriction $\mathsf{K}^\flat: C^\infty \Rightarrow C^\infty$ of $\mathsf{K}$ is an isomorphism, as $\mathsf{K}^\flat$ is the image of $\mathsf{K}$ under the equivalence $\Fin(\rVec,\rVec) \simeq [\Lin_\RR,\rVec]$ of Example \ref{exa:rvec2}.

Extending this notation, we shall write
$$\mathsf{L}^\flat,\;\mathsf{K}^\flat,\;\mathsf{J}^\flat\;:\;C^\infty \Rightarrow C^\infty$$
to denote the restrictions of the transformations $\mathsf{L},\mathsf{K},\mathsf{J}:S^\infty \Rightarrow S^\infty$ defined in Definition \ref{def:K}.

In order to show that $\mathsf{K}^\flat$ is an isomorphism, we begin by first taking a look at the coderiving transformation $\dt^\circ$ (Definition \ref{def:coderiving}) and its components $\dt^\circ_{\RR^n}$ for the finite-dimensional spaces $\RR^n$.  Recall that $\mathsf{d}^\circ_{\mathbb{R}^n}:C^\infty(\mathbb{R}^n) \otimes \mathbb{R}^n \to C^\infty(\mathbb{R}^n)$ is defined as follows: 
\[\mathsf{d}^\circ_{\mathbb{R}^n} =  (1_{C^\infty(\mathbb{R}^n)} \otimes \eta_{\mathbb{R}^n})\mathsf{m}_{\mathbb{R}^n} \]
where $\mathsf{m}_{\mathbb{R}^n}$ is the standard multiplication of $C^\infty(\mathbb{R}^n)$ and $\eta_{\mathbb{R}^n}: \mathbb{R}^n \to C^\infty(\mathbb{R}^n)$ is the linear map that sends the standard basis vectors $e_i \in \RR^n$ to the projection maps $\pi_i: \mathbb{R}^n \to \mathbb{R}$ (\ref{thm:free_cinfty}). Recalling from \ref{rem:1forms} that each element $\omega \in C^\infty(\RR^n) \otimes \RR^n$ can be expressed uniquely as a sum
 \[\omega = \sum\limits^n_{i=1} f_i \otimes e_i\;, \] 
with $f_1,...,f_n \in C^\infty(\RR^n)$, we compute that the resulting smooth function $\mathsf{d}^\circ_{\mathbb{R}^n}(\omega): \mathbb{R}^n \to \mathbb{R}$ is given by
 \begin{align*}
\mathsf{d}^\circ_{\mathbb{R}^n}(\omega)(\vec v) = \mathsf{d}^\circ_{\mathbb{R}^n}\left(\sum^n_{i=0} f_i \otimes e_i \right)(\vec v) = \sum^n_{i=0} f_i(\vec v)\pi_i (\vec v) = \sum^n_{i=0} f_i(\vec v)v_i = (f_1 (\vec v), \hdots, f_n (\vec v)) \cdot \vec v 
\end{align*}
where the symbol $\cdot$ on the right-hand side denotes the usual dot product.  Equivalently,
$$\mathsf{d}^\circ_{\mathbb{R}^n}(\omega)(\vec v) = F(\vec v) \cdot \vec v$$
where $F = \langle f_1,...,f_n \rangle:\RR^n \rightarrow \RR^n$ is the vector field corresponding to $\omega = \sum^n_{i=1} f_i \otimes e_i$ as discussed in \ref{rem:1forms}.

Now that we have computed the coderiving transformation $\dt^\circ$ for the spaces $\RR^n$, we can now explicitly describe the transformations $\mathsf{L}^\flat,\mathsf{K}^\flat,\mathsf{J}^\flat:C^\infty \Rightarrow C^\infty$.  Given a smooth function $f: \RR^n \to \RR$, we recall from \ref{rem:1forms} that the vector field corresponding to $\mathsf{d}_{\mathbb{R}^n}(f) = \sum_{i=1}^n \frac{\partial f}{\partial x_i}\otimes e_i$ is precisely the gradient $\nabla f = \left\langle \frac{\partial f}{\partial x_1},...,\frac{\partial f}{\partial x_n}\right\rangle:\RR^n \rightarrow \RR^n$ of $f$.  Using this, we obtain a simple expression for $\mathsf{L}^\flat_{\RR^n}(f) = \mathsf{L}_{\RR^n}(f): \RR^n \to \RR$:
\begin{align*}
\mathsf{L}^\flat_{\RR^n}(f)(\vec v) &= \mathsf{d}^\circ_{\mathbb{R}^n}(\mathsf{d}_{\mathbb{R}^n}(f))(\vec v) =   \nabla(f)(\vec v) \cdot \vec v
\end{align*}
By definition $\mathsf{K}^\flat_{\RR_n} = \mathsf{K}_{\RR_n} = \mathsf{L}_{\RR_n} + S^\infty(0)$ where $S^\infty(0) = C^\infty(0): C^\infty(\mathbb{R}^n) \to C^\infty(\mathbb{R}^n)$, and by Proposition \ref{thm:charn_cinfty} we deduce that $C^\infty(0)$ is given by simply evaluating at zero:
\[C^\infty(0)(f)(\vec v) = f(\vec 0)\;.\]
It is interesting to note that $C^\infty(0)= \mathsf{e}_{\RR^n} \mathsf{u}_{\RR^n}$, where $\mathsf{e}_{\RR^n}$ is the counit map defined in Section \ref{codersec}. Therefore, $\mathsf{K}_{\mathbb{R}^n}^\flat(f) = \mathsf{K}_{\RR^n}(f): \mathbb{R}^n \to \mathbb{R}$ is given by
\[\mathsf{K}_{\mathbb{R}^n}^\flat(f)(\vec v)=  \nabla(f)(\vec v) \cdot \vec v + f(\vec 0)\;.\] 
Lastly, we find that $\mathsf{J}_{\mathbb{R}^n}^\flat(f) = \mathsf{J}_{\mathbb{R}^n}(f): \mathbb{R}^n \to \mathbb{R}$ is given as follows:
\[\mathsf{J}_{\mathbb{R}^n}^\flat(f)(\vec v)=  \nabla(f)(\vec v) \cdot \vec v + f(\vec v)\;.\] 

We wish to show that $\mathsf{K}^\flat$ and $\mathsf{J}^\flat$ are natural isomorphisms. To do so, we need to make use of the \emph{Fundamental Theorem of Line Integration}, which relates the gradient and line integration. Recall that for any continuous map $F:\RR^n \rightarrow \RR^n$ and a curve $C$ parametrized by a given smooth path $r: [a,b] \to \RR^n$, the \textbf{line integral of $F$ along $C$} is defined as
\[\int \limits_C F\cdot \mathsf{d}r := \int \limits \limits^b_a F(r(t)) \cdot r^\prime (t) ~ \mathsf{d}t \;.\]
The \textbf{Fundamental Theorem of Line Integration} states that for every $C^1$ function $f: \RR^n \to \RR$ and any smooth path $r: [a,b] \to \RR^n$, we have the following equality: 
\[\int \limits_C \nabla f \cdot \dt r = f(r(t)) \Big|^b_a = f(r(b)) - f(r(a)) \]
Note that the Fundamental Theorem of Line Integration is a higher-dimensional generalization of the Second Fundamental Theorem of Calculus. 

Another basic tool that we will require is a compatibility relation between integration and differentiation called the \textbf{Leibniz integral rule}, to the effect that partial differentiation and integration commute when they act on independent variables: 
 \begin{align*}
 \frac{\partial}{\partial x} \int \limits^b_af(x, t) ~\mathsf{d}t = \int \limits^b_a\frac{\partial f(x, t)}{\partial x} ~\mathsf{d}t 
 \end{align*}
for any $C^1$ function $f:\RR^2 \rightarrow \RR$ and any constants $a,b \in \RR$, noting that this equation also holds under more general hypotheses, such as those in  \cite[\S 8.1, Thm. 1]{PrMo:IntCal}.  As a consequence, if $f:\RR^n \times \RR \rightarrow \RR$ is a $C^1$ function, which we shall write as a function $f(\vec x,t)$ of a vector variable $\vec x$ and a scalar variable $t$, then any pair of constants $a,b \in \RR$ determines a $C^1$ function $g(\vec x) = \int_a^b f(\vec x,t) ~ \mathsf{d}t$ of $\vec x$ whose gradient $\nabla g:\RR^n \rightarrow \RR^n$ is 
$$\nabla\left(\int\limits_a^b f(\vec x,t) ~ \mathsf{d}t\right) = \int\limits_a^b \nabla(f(\vec x,t)) ~ \mathsf{d} t\;,$$
where the right-hand side is an $\RR^n$-valued integral and is regarded as a function of $\vec x \in \RR^n$.  As will be our convention throughout the sequel, the gradient in each case is taken with respect to the variable $\vec x$.

\begin{proposition}\label{thm:kflat_nat_iso} $\mathsf{K}^\flat: C^\infty \Rightarrow C^\infty$ is a natural isomorphism. 
\end{proposition}
\begin{proof} For each $n$, define $\mathsf{K}^*_{\RR^n}: C^\infty(\RR^n) \to C^\infty(\RR^n)$ as follows:
\[\mathsf{K}^*_{\RR^n}(f)(\vec v)= \int \limits^1_0 \int \limits^1_0 \nabla(f)(st \vec v) \cdot \vec v ~\mathsf{d}s ~\mathsf{d}t + f(\vec 0)\;,\]
for each $f \in C^\infty(\RR^n)$ and $\vec v \in \RR^n$, noting that the resulting function $\mathsf{K}^*_{\RR^n}(f):\RR^n \rightarrow \RR$ is smooth, as a consequence of the Leibniz integral rule.  We shall use the Fundamental Theorem of Line Integration to show that $\mathsf{K}^*_{\RR^n}$ is an inverse of $\mathsf{K}^\flat_{\mathbb{R}^n} = \mathsf{K}_{\mathbb{R}^n}$. 

For each $\vec v \in \mathbb{R}^n$, define the smooth path $r_{\vec v}: [0,1] \to \mathbb{R}^n$  by
\[r_{\vec v}(t) =   t\vec v\;.\]
Note that $r_{\vec v}$ is a parametrization of the straight line between $\vec 0$ and $\vec v$, which we will denote as $C_{\vec v}$. The derivative of $r_{\vec v}$ is simply the constant function that maps everything to $\vec v$: $r^\prime_{\vec v}(t) = \vec v$. Now the Fundamental Theorem of Line Integration implies that for every smooth function $f: \mathbb{R}^n \to \mathbb{R}$, the following equality holds: 
\begin{equation}\label{FTL}\begin{gathered}  \int \limits^1_0 \nabla(f)(t \vec v) \cdot \vec v ~\mathsf{d}t =  \int \limits^1_0 \nabla(f)(r_{\vec v}(t)) \cdot r^\prime_{\vec v}(t) ~\mathsf{d}t =  \int \limits_{C_{\vec v}} \! \nabla f \cdot dr = f(r_{\vec v}(1)) - f(r_{\vec v}(0)) = f(\vec v) - f(\vec 0)\;. \end{gathered}\end{equation}
This is the key identity to the proof that $\mathsf{K}^\flat$ is an isomorphism. In fact, later we will see that the Second Fundamental Theorem of Calculus rule {\bf [c.1]} from the definition of a co-calculus category (Definition \ref{cocaldef}) is precisely this instance of the Fundamental Theorem of Line Integration. 

Now observe that for any smooth function $f:\RR^n \rightarrow \RR$, the following equality holds: 
\[\mathsf{K}_{\mathbb{R}^n}(f)(\vec 0)= \underbrace{\nabla(f)(\vec 0) \cdot \vec 0}_{=~0} + f(\vec 0) = f(\vec 0)\]
Using the Fundamental Theorem of Line Integration twice and playing with the bounds of the integral using limits, we show that $\mathsf{K}_{\mathbb{R}^n}\mathsf{K}^*_{\RR^n}=1$: 
\begin{align*}
\mathsf{K}^*_{\mathbb{R}^n}(\mathsf{K}_{\mathbb{R}^n}(f))(\vec v) &= \int \limits^1_0 \int \limits^1_0 \nabla(\mathsf{K}_{\mathbb{R}^n}(f))(st \vec v) \cdot v ~\mathsf{d}s ~\mathsf{d}t + \mathsf{K}_{\mathbb{R}^n}(f)(\vec 0) \\
&= \lim_{u \to 0^+} \int \limits^1_u \int \limits^1_0 \nabla(\mathsf{K}_{\mathbb{R}^n}(f))(st \vec v) \cdot v ~\mathsf{d}s ~\mathsf{d}t + f(\vec 0) \\
&= \lim_{u \to 0^+} \int \limits^1_u \frac{t}{t} \int \limits^1_0 \nabla(\mathsf{K}_{\mathbb{R}^n}(f))(st \vec v) \cdot \vec v ~\mathsf{d}s ~\mathsf{d}t + f(\vec 0) \\
&=  \lim_{u \to 0^+} \int \limits^1_u \frac{1}{t} \int \limits^1_0 \nabla(\mathsf{K}_{\mathbb{R}^n}(f))(st \vec v) \cdot t \vec v ~\mathsf{d}s ~\mathsf{d}t + f(\vec 0) \\
&=\lim_{u \to 0^+} \int \limits^1_u \frac{1}{t} \left(\mathsf{K}_{\mathbb{R}^n}(f)(t\vec v) - \mathsf{K}_{\mathbb{R}^n}(f)(\vec 0)\right) ~\mathsf{d}t + f(\vec 0)\\
&= \lim_{u \to 0^+} \int \limits^1_u \frac{1}{t} \left(\nabla(f)(t\vec v) \cdot t\vec v +f(\vec 0) - f(\vec 0) \right) ~\mathsf{d}t + f(\vec 0)\\
&= \lim_{u \to 0^+} \int \limits^1_u \frac{1}{t}~ \nabla(f)(t\vec v) \cdot t\vec v ~\mathsf{d}t + f(\vec 0)\\
&= \lim_{u \to 0^+} \int \limits^1_u\frac{t}{t}~ \nabla(f)(t\vec v) \cdot \vec v ~\mathsf{d}t + f(\vec 0)\\
&= \lim_{u \to 0^+} \int \limits^1_u \nabla(f)(t\vec v) \cdot \vec v ~\mathsf{d}t + f(\vec 0)\\
&= \int \limits^1_0 \nabla(f)(t\vec v) \cdot \vec v ~\mathsf{d}t + f(\vec 0)\\
&= f(\vec v) - f(\vec 0) + f(\vec 0) \\
&= f(\vec v).  
\end{align*}
To prove that $\mathsf{K}^*_{\mathbb{R}^n}\mathsf{K}_{\mathbb{R}^n}=1$, we will begin with a few preliminary observations. First, the following identities also hold for any smooth function $f:\RR^n \rightarrow \RR$:
\[\mathsf{K}^*_{\mathbb{R}^n}(f)(\vec 0)= \underbrace{\int \limits^1_0 \int \limits^1_0 \nabla(f)(st \vec 0) \cdot \vec 0 ~\mathsf{d}s ~\mathsf{d}t}_{=~0} + f(\vec 0)=f(\vec 0).  \]
Next, as a consequence of the gradient version of the chain rule, for any scalar $t \in \RR$ we have that
\[\nabla(f(t \vec x)) (\vec v) \cdot \vec w = \nabla(f)(t \vec v) \cdot t\vec w, \]
where the gradient on the left-hand side is taken with respect to the variable $\vec x$ and then explicitly evaluated at $\vec v$; we will use similar notation in the sequel.

The last observation we need is that the gradient interacts nicely with our line integral, as a consequence of the Leibniz integral rule and the chain rule:   
\begin{equation}\label{gradientLineCommutes}\nabla \left( \int \limits^1_0 f(t \vec x) ~\mathsf{d}t\right)(\vec v) \cdot \vec w = \int \limits^1_0 \nabla (f(t \vec x))(\vec v) \cdot \vec w~\mathsf{d}t= \int \limits^1_0 \nabla (f)(t \vec v) \cdot t \vec w~\mathsf{d}t
\end{equation}

With all these observations and using similar techniques from before, we can prove that $\mathsf{K}^*_{\mathbb{R}^n}\mathsf{K}_{\mathbb{R}^n}=1$: 
\begin{align*}
\mathsf{K}_{\mathbb{R}^n}(\mathsf{K}^*_{\mathbb{R}^n}(f))(\vec v) &= \nabla(\mathsf{K}^*_{\mathbb{R}^n}(f))(\vec v) \cdot \vec v + \mathsf{K}^*_{\mathbb{R}^n}(f)(\vec 0) \\
&= \nabla\left(\int \limits^1_0 \int \limits^1_0 \nabla(f)(st \vec x) \cdot \vec x ~\mathsf{d}s ~\mathsf{d}t + f(\vec 0) \right)(\vec v) \cdot \vec v + f(\vec 0) \\ 
&=  \nabla\left(\int \limits^1_0 \int \limits^1_0 \nabla(f)(st \vec x) \cdot \vec x ~\mathsf{d}s ~\mathsf{d}t \right)(\vec v) \cdot \vec v +  \underbrace{ \nabla(f(\vec 0))(\vec v) \cdot \vec v}_{=~0} + f(\vec 0) \\
&= \nabla\left(\lim_{u \to 0^+} \int \limits^1_u \int \limits^1_0 \nabla(f)(st \vec x) \cdot \vec x ~\mathsf{d}s ~\mathsf{d}t \right)(\vec v) \cdot \vec v + f(\vec 0) \\
&= \nabla\left(\lim_{u \to 0^+} \int \limits^1_u \frac{t}{t} \int \limits^1_x \nabla(f)(st \vec x) \cdot \vec x ~\mathsf{d}s ~\mathsf{d}t \right)(\vec v) \cdot \vec v + f(\vec 0) \\
&= \nabla\left(\lim_{u \to 0^+}\int \limits^1_u \frac{1}{t} \int \limits^1_x \nabla(f)(st \vec x) \cdot t \vec x ~\mathsf{d}s ~\mathsf{d}t \right)(\vec v) \cdot \vec v + f(\vec 0) \\
&= \nabla\left(\lim_{u \to 0^+} \int \limits^1_u \frac{1}{t} \left(f(t\vec x) - f(\vec 0) \right)  ~\mathsf{d}t \right)(\vec v) \cdot \vec v + f(\vec 0) \\
&= \nabla\left(\lim_{u \to 0^+} \int \limits^1_u \frac{1}{t} f(t\vec x) ~\mathsf{d}t \right)(\vec v) \cdot \vec v - \nabla\left(\lim_{u \to 0^+} \int \limits^1_u \frac{1}{t} f(\vec 0)  ~\mathsf{d}t \right)(\vec v) \cdot \vec v + f(\vec 0) \\
&= \lim_{u \to 0^+} \int \limits^1_u \frac{1}{t} \nabla(f(t\vec x))(\vec v) \cdot \vec v ~\mathsf{d}t - \lim_{u \to 0^+} \int \limits^1_u \frac{1}{t} \underbrace{\nabla(f(\vec 0))(\vec v) \cdot \vec v}_{=~0} ~\mathsf{d}t  + f(\vec 0) \\
&= \lim_{u \to 0^+} \int \limits^1_u \frac{1}{t} \nabla(f)(t\vec v) \cdot t\vec v ~\mathsf{d}t + f(\vec 0) \\
&= \lim_{u \to 0^+} \int \limits^1_u \frac{t}{t} \nabla(f)(t\vec v) \cdot \vec v ~\mathsf{d}t + f(\vec 0) \\
&= \lim_{u \to 0^+} \int \limits^1_u \nabla(f)(t\vec v) \cdot \vec v ~\mathsf{d}t + f(\vec 0) \\
&= \int \limits^1_0 \nabla(f)(t\vec v) \cdot \vec v ~\mathsf{d}t + f(\vec 0)\\
&= f(\vec v) - f(\vec 0) + f(\vec 0) \\
&= f(\vec v)
\end{align*}
Therefore we conclude that $\mathsf{K}^\flat: C^\infty \Rightarrow C^\infty$ is a natural isomorphism. 
\end{proof}

\begin{corollary} $\mathsf{K}: \mathsf{S}^\infty \Rightarrow \mathsf{S}^\infty$ is a natural isomorphism. 
\end{corollary}

Therefore by Theorem \ref{antithm}, we obtain the following result: 

\begin{theorem}
The monad $\SSS^\infty$ on $\rVec$ has the structure of a codifferential category with antiderivatives and therefore also the structure of a co-calculus category.
\end{theorem}

Before giving an explicit description of the induced integral transformation $\mathsf{s}$, we take a look at the inverse of $\mathsf{J}$. Using $\mathsf{J}^{-1}$ will simplify calculating $\mathsf{s}$. Recall that $\mathsf{K}$ being a natural isomorphism implies that $\mathsf{J}$ is a natural isomorphism.  One can then construct $\mathsf{J}^{-1}$ from $\mathsf{K}^{-1}$ with the aid of $\mu$ \cite{integralAndCalcCats}.  However, in the present case, for the finite-dimensional spaces $\RR^n$, we will see that $\mathsf{J}^{-1}_{\RR^n}$ can be described by a considerably simpler formula that our integral formula for $\mathsf{K}^{-1}_{\RR^n} = \mathsf{K}^*_{\RR^n}$ in Proposition \ref{thm:kflat_nat_iso}.  For this reason, and for the sake of completeness, we will give a stand-alone proof that $\mathsf{J}$ is invertible, by directly defining an inverse of $\mathsf{J}_{\RR^n}$ by means of an integral formula.

\begin{proposition}\label{thm:jflat_nat_iso} $\mathsf{J}^\flat: C^\infty \Rightarrow C^\infty$ is a natural isomorphism. 
\end{proposition} 
\begin{proof} For each $n$, define $\mathsf{J}^*_{\mathbb{R}^n}: C^\infty(\RR^n) \to C^\infty(\RR^n)$ as follows: 
\[\mathsf{J}^*_{\mathbb{R}^n}(f)(\vec v)= \int \limits^1_0 f(t \vec v) ~\mathsf{d}t\;,\]
noting that the Leibniz integral rule entails that $\mathsf{J}^*_{\RR^n}(f)$ is indeed smooth.  Again, we wish to use the Fundamental Theorem of Line Integration to show that this is indeed the inverse of $\mathsf{J}^\flat_{\mathbb{R}^n} = \mathsf{J}_{\RR^n}$. 

Given a smooth function $f: \RR^n \to \RR$, define the smooth function $\tilde{f}: \RR^n \times \RR \to \RR$ simply as multiplying $f$ by a scalar: $\tilde{f}(\vec v, t) = t f(\vec v)$. Its gradient $\nabla (\tilde{f}): \RR^n \times \RR \to \RR^n \times \RR$ is given by
\begin{align*}
\nabla (\tilde{f}) (\vec v, t) &=~ \left( \frac{\partial \tilde{f}}{\partial x_1}(\vec v, t), \frac{\partial \tilde{f}}{\partial x_1}(\vec v, t), \hdots, \frac{\partial \tilde{f}}{\partial x_n}(\vec v, t), \frac{\partial \tilde{f}}{\partial t}(\vec v, t) \right) \\
&=~ \left( t \frac{\partial f}{\partial x_1}(\vec v) , \hdots, t \frac{\partial f}{\partial x_n}(\vec v) , f(\vec v) \right)\\
&=~ (t \nabla(f) (\vec v), f(\vec v))\;.
\end{align*}
As a consequence, we obtain the following identities: 
\[\nabla (\tilde{f}) (\vec v, t) \cdot (\vec w, 1) = t \nabla(f) (\vec v) \cdot \vec w + f(\vec v) = \nabla(f)(\vec v) \cdot t \vec w + f(\vec v) = \nabla (\tilde{f}) (\vec v, 1) \cdot (t \vec w, 1)\;. \]
Now using this above identity and the Fundamental Theorem of Line Integration, we show that $\mathsf{J}_{\mathbb{R}^n}\mathsf{J}^*_{\mathbb{R}^n}=1$: 
\begin{align*}
\mathsf{J}^*_{\mathbb{R}^n}\left (\mathsf{J}_{\mathbb{R}^n}(f) \right)(\vec v) &=  \int \limits^1_0 \mathsf{J}_{\mathbb{R}^n}^\flat(f)(t \vec v) ~ \dt t  \\
&=  \int \limits^1_0 \left( \nabla(f)(t \vec v) \cdot t \vec v + f(t\vec v) \right)  ~ \dt t \\
&= \int \limits^1_0 \left( \nabla (\tilde{f}) (t \vec v, t) \cdot (\vec v, 1) \right)  ~ \dt t \\
&= \tilde{f}(\vec v, 1) - \tilde{f}(\vec 0, 0) \\
&=f(\vec v) 
\end{align*}

Having thus shown that $\mathsf{J}^*_{\RR^n}$ is a retraction of $\mathsf{J}_{\RR^n}$, and having already noted above that $\mathsf{J}$ is invertible since $\mathsf{K}$ is invertible, we may at this point deduce that $\mathsf{J}^*_{\RR^n} = \mathsf{J}_{\RR^n}^{-1}$.  However, in order to construct a standalone proof that $\mathsf{J}$ is invertible, we now show directly that $\mathsf{J}^*_{\mathbb{R}^n}\mathsf{J}_{\mathbb{R}^n}=1$, by using the interchange identity between the gradient and the line integral \eqref{gradientLineCommutes}: 
\begin{align*}
\mathsf{J}_{\mathbb{R}^n}\left (\mathsf{J}^*_{\mathbb{R}^n}(f) \right)(\vec v) &=  \nabla\left( \mathsf{J}^*_{\mathbb{R}^n}(f) \right) (\vec v) \cdot  \vec v + \mathsf{J}^*_{\mathbb{R}^n}(f)(\vec v) \\
&= \nabla\left( \int \limits^1_0 f(t \vec x) ~\mathsf{d}t \right) (\vec v) \cdot \vec v + \int \limits^1_0 f(t \vec v) ~\mathsf{d}t \\
&= \int \limits^1_0 \nabla(f) (t \vec v) \cdot t \vec v ~\mathsf{d}t + \int \limits^1_0 f(t \vec v) ~\mathsf{d}t \\
&=  \int \limits^1_0 \left( \nabla(f)(t \vec v) \cdot t \vec v + f(t\vec v) \right)  ~ \dt t \\
&= \int \limits^1_0 \left( \nabla (\tilde{f}) (t \vec v, t) \cdot (\vec v, 1) \right)  ~ \dt t \\
&= \tilde{f}(\vec v, 1) - \tilde{f}(\vec 0, 0) \\
&=f(\vec v) 
\end{align*}
\end{proof} 

\begin{corollary} $\mathsf{J}: \mathsf{S}^\infty \Rightarrow \mathsf{S}^\infty$ is a natural isomorphism. 
\end{corollary}

We now compute the induced integral transformation $\mathsf{s}$ for the finite-dimensional vector spaces $\RR^n$, that is, we compute a formula for the map
$$\mathsf{s}_{\RR^n}:C^\infty(\mathbb{R}^n) \otimes \mathbb{R}^n \longrightarrow C^\infty(\mathbb{R}^n)\;,$$
which we recall is defined as $\mathsf{s}_{\mathbb{R}^n} =  \mathsf{d}^\circ_{\RR^n}{\mathsf{K}}_{\mathbb{R}^n}^{-1} = (\mathsf{J}^{-1}_{\mathbb{R}^n} \otimes 1_{\RR^n})\mathsf{d}^\circ_{\RR^n}$ (Theorem \ref{antithm}).  Given any element $\omega = \sum_{i=1}^n f_i \otimes e_i$ of $C^\infty(\RR^n) \otimes \RR^n$, expressed as in \ref{rem:1forms}, we compute that
\begin{align*}
\mathsf{s}_{\mathbb{R}^n}(\omega)(\vec v) &= { \mathsf{d}^\circ} \left ( \left(\mathsf{J}^{-1}_{\mathbb{R}^n} \otimes 1_{\RR^n} \right) \left(\sum_{i=1}^n f_i \otimes e_i \right) \right)(\vec v) \\
&=  {\mathsf{d}^\circ} \left(\sum_{i=1}^n \left(\int \limits^1_0 f_i(t \vec x) ~\mathsf{d}t \right) \otimes e_i \right)(\vec v) \\
&= \left( \int \limits^1_0 f_1(t \vec v) ~\mathsf{d}t, \hdots, \int \limits^1_0 f_n(t \vec v) ~\mathsf{d}t \right) \cdot \vec v \\
&= \int \limits^1_0 (f_1 (t \vec v), \hdots, f_n (t \vec v)) \cdot \vec v ~\mathsf{d}t \\
&=  \int \limits^1_0 F(t \vec v) \cdot \vec v ~\mathsf{d}t
\end{align*}
where $F: \RR^n \to \RR^n$ is the vector field $F = \langle f_1, \hdots, f_n\rangle$ corresponding to $\omega$ as in \ref{rem:1forms}. Equivalently, $\mathsf{s}_{\RR}(\omega)(\vec v)$ can be described as the line integral
$$\mathsf{s}_{\mathbb{R}^n}(\omega)(\vec v) = \int\limits_{C_{\vec v}} F \cdot \mathsf{d}r$$
of the vector field $F$ along the directed line segment $C_{\vec v}$ from the origin to the point $\vec v$ (for which one parametrization is $r = r_{\vec v}$, as discussed in the proof of Proposition \ref{thm:kflat_nat_iso}).  Recalling that $\omega$ is a 1-form on $\RR^n$ (\ref{rem:1forms}), this line integral is more succinctly described as follows:

\begin{theorem}\label{thm:1formint}
The free $C^\infty$-ring modality $\mathsf{S}^\infty$ carries the structure of an integral category.  The integral transformation $\mathsf{s}$ carried by $\mathsf{S}^\infty$ sends each 1-form $\omega \in C^\infty(\RR^n) \otimes \RR^n$ to the function $\mathsf{s}_{\RR^n}(\omega) \in C^\infty(\RR^n)$ whose value at each $\vec v \in \RR^n$ is the integral of $\omega$ along the directed line segment $C_{\vec v}$ from $\vec 0$ to $\vec v$:
$$\mathsf{s}_{\mathbb{R}^n}(\omega)(\vec v) = \int\limits_{C_{\vec v}} \omega\;.$$
\end{theorem}

\begin{remark}\label{rem:1formint}
For brevity, we will write the integral $\int_{C_{\vec v}}\omega$ in Theorem \ref{thm:1formint} as $\int_{\vec v}\omega$, as it can be thought of as an integral \textit{over $\vec v$, considered as a position vector}.  Correspondingly, we will denote the function $\mathsf{s}_{\RR^n}(\omega):\RR^n \rightarrow \RR$ by $\int_{(\text{-})}\omega$.
\end{remark}

\begin{example}\label{rmk:integralOfPolys}
It is illustrative to consider what the above formulae produce when the input is a 1-form $\omega$ with polynomial coefficients.  For example, writing $\vec x = (x_1,x_2)$ for a general point in $\RR^2$, let $\omega$ be the 1-form $\omega = x_1^2x_2^5\,\mathsf{d}x_1 + x_1^3\,\mathsf{d}x_2$ on $\RR^2$ (with the notation of \ref{rem:1forms}), whose corresponding vector field $F$ is given by $F(x_1,x_2) = (x_1^2x_2^5,x_1^3)$.  Then $F(tx_1,tx_2) = ((tx_1)^2(tx_2)^5,(tx_1)^3) = (t^7x_1^2x_2^5,t^3x_1^3)$ so that $\mathsf{s}_{\RR^n}(\omega)(\vec x)$ is the integral
\[\int\limits_{\vec x}\omega = \int \limits^1_0 F(t \vec x) \cdot \vec x ~\mathsf{d}t = \int_0^1(t^7x_1^2x_2^5x_1 + t^3x_1^3x_2) ~\mathsf{d}t = \frac{1}{8}x_1^3x_2^5 + \frac{1}{4}x_1^3x_2. \]
More generally, one can readily show that when applied to any 1-form
\[\omega = \sum_{i=1}^n p_i\,\mathsf{d}x_i = \sum_{i=1}^n p_i\otimes e_i\]
 on $\RR^n$ with polynomial coefficients $p_i$, the above formulae for $\mathsf{s}$ reproduce the integral transformation for polynomials as described in Example \ref{ex:rVecInt}.  The formula for arbitrary smooth functions thus explains the seemingly odd choice of summing all the coefficients when integrating a particular term.  
\end{example}

Let us now examine what the identities of a co-calculus category (Definition \ref{cocaldef}) amount to in the specific co-calculus category that we have developed here. The Second Fundamental Theorem of Calculus rule {\bf [c.1]} is precisely the special case of the Fundamental Theorem of Line Integration that we used extensively in the proofs of Propositions \ref{thm:kflat_nat_iso} and \ref{thm:jflat_nat_iso}, namely \eqref{FTL}. Indeed, given a smooth function $f \in C^\infty(\RR^n)$, one has that 
 \begin{align*}
\mathsf{s}_{\mathbb{R}^n}\left( \mathsf{d}_{\mathbb{R}^n}(f) \right)(\vec v) + S^\infty(0)(f)(\vec v) = \mathsf{s}_{\mathbb{R}^n}\left(  \sum_{i=1}^n \frac{\partial f}{\partial x_i}\otimes e_i \right)(\vec v) + f(\vec 0) =  \int \limits^1_0 \nabla(f) (t \vec v) \cdot \vec v ~\mathsf{d}t + f(\vec 0) = f(\vec v)\:.
\end{align*}
On the other hand the Poincar\'{e} condition {\bf [c.2]} is essentially the statement of its namesake, the Poincar\'{e} Lemma, for 1-forms on Euclidean spaces. Explicitly, {\bf [c.2]} says that closed 1-forms are exact and that the integral transformation $\mathsf{s}$ provides a canonical choice of 0-form to serve as `antiderivative' for each closed 1-form. So if $\omega$ is a closed 1-form over $\mathbb{R}^n$, then $\omega$ is exact by being the exterior derivative of the 0-form $\mathsf{s}_{\mathbb{R}^n}(\omega)$, that is, $\mathsf{d}_{\mathbb{R}^n}(\mathsf{s}_{\mathbb{R}^n}(\omega)) = \omega$. 

We now take a look at the Rota-Baxter rule \textbf{[s.2]} for the integral transformation $\mathsf{s}$ (Definition \ref{defn:coint}). Continuing to identify $C^\infty(\RR^n) \otimes \RR^n$ with the $C^\infty(\RR^n)$-module of smooth 1-forms on $\RR^n$ as in Remark \ref{rem:1forms}, we will employ the usual notation $f\omega$ for the product of a function $f \in C^\infty(\RR^n)$ and a 1-form $\omega$.  Given two 1-forms $\omega,\nu \in C^\infty(\RR^n) \otimes \RR^n$, the Rota-Baxter rule \textbf{[s.2]} gives the following equality, with the notation of Remark \ref{rem:1formint}: 
\[\left(\int\limits_{\vec v}\omega\right)\left(\int\limits_{\vec v}\nu\right)  \;\;\;\;=\;\;\;\;  \int\limits_{\vec v}\left(\int\limits_{(\text{-})}\nu\right)\omega \;\;+\;\; \int\limits_{\vec v}\left(\int\limits_{(\text{-})}\omega\right)\nu\]

The Rota-Baxter identity also admits a nice (and possibly more explicit) expression in terms of vector fields. Indeed, given two vector fields $F: \RR^n \to \RR^n$ and $G: \RR^n \to \RR^n$, then the Rota-Baxter rule \textbf{[s.2]} implies that the following equality holds: 

\begin{align*}
 &\left( \int \limits^1_0 F(t \vec v) \cdot \vec v ~\mathsf{d}t  \right) \left( \int \limits^1_0 G(t\vec v) \cdot \vec v ~\mathsf{d}t  \right) = \\
 &\int \limits^1_0 \left( F(t \vec v) \cdot  \vec v \right) \left( \int \limits^t_0 G(u \vec v) \cdot \vec v ~\mathsf{d}u \right) ~\mathsf{d}t  +  \int \limits^1_0 \left( \int \limits^t_0 F (u \vec v)\cdot \vec v ~\mathsf{d}u  \right) \left( G(t \vec v) \cdot   \vec v \right)~\mathsf{d}t 
\end{align*}

A further consequence of the Rota-Baxter rule, for arbitrary vector spaces $V$, is that the integral transformation $\mathsf{s}_V: S^\infty(V) \otimes V \to S^\infty(V)$ induces a \textit{Rota-Baxter operator} on the free $C^\infty$-ring $S^\infty(V)$, as we will show in Proposition \ref{prop:rotaBaxterAlgebras}. 

\begin{definition}  Let $R$ be a commutative ring. A (commutative) \textbf{Rota-Baxter algebra} \cite{guo2012introduction} (of weight $0$) over $R$ is a pair $(A, \mathsf{P})$ consisting of a (commutative) $R$-algebra $A$ and an $R$-linear map $\mathsf{P}: A \to A$ such that $\mathsf{P}$ satisfies the Rota-Baxter identity; that is, for each $a,b \in A$, the following equality holds:
\begin{equation}\label{eq:rb_alg}\mathsf{P}(a)\mathsf{P}(b)=\mathsf{P}(a\mathsf{P}(b))+\mathsf{P}(\mathsf{P}(a)b). \end{equation}
 The map $\mathsf{P}$ is called a \textbf{Rota-Baxter operator}.
\end{definition}

As discussed in \cite{integralAndCalcCats}, the latter Rota-Baxter identity \eqref{eq:rb_alg} corresponds to a formulation of the integration by parts rule that involves only integrals and no derivatives---as we will soon illustrate in Example \ref{ex:smoothRB}. We refer the reader to \cite{guo2012introduction} for more details on Rota-Baxter algebras. 

Now for an arbitrary $\mathbb{R}$-vector space $V$ and any element $\vec v \in V$, it readily follows from the Rota-Baxter rule \textbf{[s.2]} in Definition \ref{defn:coint} that the corresponding linear map $v: \mathbb{R} \to V$ induces a Rota-Baxter operator $\mathsf{P}_v: S^\infty(V) \to S^\infty(V)$ defined as the following composite
\begin{equation}\label{RB}\begin{gathered}  \mathsf{P}_v := \xymatrixcolsep{5pc}\xymatrix{ S^\infty(V) \ar[r]^-{1 \otimes v} &  S^\infty(V) \otimes V \ar[r]^-{\mathsf{s}_V} & S^\infty(V)\;,
  } \end{gathered}\end{equation}
making the pair $(S^\infty(V), \mathsf{P}_v)$ a Rota-Baxter algebra over $\mathbb{R}$. Summarizing, we obtain the following new observation: 

\begin{proposition}\label{prop:rotaBaxterAlgebras}
Free $C^\infty$-rings are commutative Rota-Baxter algebras over $\mathbb{R}$, with Rota-Baxter operators defined as in (\ref{RB}).
\end{proposition} 

\begin{example}\label{ex:smoothRB}
A particularly important example arises when we let $V = \mathbb{R}$ and we take $\vec v$ to be the element $1 \in \mathbb{R}$ (whose corresponding linear map is the identity on $\mathbb{R}$). In this case, the corresponding Rota-Baxter operator $ \mathsf{P}_1$ on $S^\infty(\mathbb{R}) = C^\infty(\mathbb{R})$ is essentially the integral transformation: 
$$  \mathsf{P}_1 := \xymatrixcolsep{5pc}\xymatrix{ C^\infty(\mathbb{R}) \ar[r]^-{\cong} &  C^\infty(\mathbb{R}) \otimes \mathbb{R} \ar[r]^-{\mathsf{s}_V} & C^\infty(\mathbb{R})\;.
  } $$
Letting $f \in C^\infty(\mathbb{R})$, we can use the substitution rule to compute that the function $\mathsf{P}_1(f) \in C^\infty(\mathbb{R})$ is given by 
$$ \mathsf{P}_1(f)(x) = \mathsf{s}_{\RR}(f \otimes 1)(x) =  \int \limits^1_0 f(t x) x ~\mathsf{d}t = \int \limits^x_0 f(u) ~\mathsf{d}u\;.$$  
Expressed in this form, the Rota-Baxter algebra $(C^\infty(\mathbb{R}), \mathsf{P}_1)$ is often considered the canonical example of a Rota-Baxter algebra (of weight $0$). For a  pair of smooth functions $f,g \in C^\infty(\mathbb{R})$, the Rota-Baxter identity is 
\begin{align*}
 \mathsf{P}_1(f)(x) \mathsf{P}_1(g)(x) &= \left( \int \limits^x_0 f(u) ~\mathsf{d}u \right) \left( \int \limits^x_0 g(u) ~\mathsf{d}u \right) \\
 &= \int \limits^x_0 f(u) \left( \int \limits^u_0 g(t) ~\mathsf{d}t \right) ~\mathsf{d}u + \int \limits^x_0  \left( \int \limits^u_0 f(t) ~\mathsf{d}t \right)g(u) ~\mathsf{d}u \\
 &= \mathsf{P}_1\left(f \mathsf{P}_1(g) \right)(x) + \mathsf{P}_1\left( \mathsf{P}_1(f) g \right)\;.
\end{align*}
\end{example}

One interesting consequence of Rota-Baxter algebra structure is that the Rota-Baxter operator induces a new \emph{non-unital} Rota-Baxter algebra structure. If $(A, \mathsf{P})$ is a Rota-Baxter algebra over $R$, then define a new associative binary operation $\ast_\mathsf{P}$ by 
$$ a \ast_\mathsf{P} b =  a\mathsf{P}(b) +  \mathsf{P}(a)b
\;.$$
This new multiplication $\ast_\mathsf{P}$ is called the double product and endows $A$ with a non-unital $R$-algebra structure, with respect to which $\mathsf{P}$ is again a Rota-Baxter operator. If $A$ is commutative, then the double product is also commutative. Also note that by $R$-linearity of $\mathsf{P}$, the Rota-Baxter identity can then be re-expressed as:
$$\mathsf{P}(a \ast_\mathsf{P} b) = \mathsf{P}(a)\mathsf{P}(b)$$
which implies that $\mathsf{P}$ is a non-unital Rota-Baxter algebra homomorphism. 

\begin{corollary} In addition to its underlying unital $\RR$-algebra structure, each free $C^\infty$-ring carries a further non-unital, commutative $\RR$-algebra structure, with the same addition operation but with multiplication given by the double product induced by the Rota-Baxter operator defined in \eqref{RB}.
\end{corollary}

\begin{example} \normalfont Consider the Rota-Baxter algebra $(C^\infty(\mathbb{R}), \mathsf{P}_1)$ from Example \ref{ex:smoothRB}. In this case, the induced double product $\ast_{\mathsf{P}_1}$ is given by 
\[(f \ast_{\mathsf{P}_1} g)(x) =  f(x) \left( \int \limits^x_0 g(t) ~\mathsf{d}t \right) +  \left( \int \limits^x_0 f(t) ~\mathsf{d}t \right) g(x) =  f(x)\mathsf{P}_1(g)(x) +  \mathsf{P}_1(f)(x)g(x)\;.\]
\end{example}

\bibliographystyle{amsplain}
\bibliography{smoothTensorIntegral}

\end{document}